\documentclass[reqno,11pt]{amsart}
\usepackage{amsmath,amssymb,latexsym,soul,cite,mathrsfs,accents}
\usepackage{color,enumitem,graphicx}
\usepackage{xcolor}
\usepackage[colorlinks=true,urlcolor=blue,
citecolor=red,linkcolor=blue,linktocpage,pdfpagelabels,
bookmarksnumbered,bookmarksopen]{hyperref}
\usepackage[english]{babel}
\usepackage[left=2.6cm,right=2.6cm,top=2.9cm,bottom=2.9cm]{geometry}
\usepackage{tensor}
\usepackage{tikz,pgfplots}
\pgfplotsset{compat=1.18}
\usepackage{cancel}
\usepackage{esint}
\newtheorem{theorem}{Theorem}
\newtheorem{lemma}[theorem]{Lemma}
\newtheorem{corollary}[theorem]{Corollary}
\newtheorem{proposition}[theorem]{Proposition}
\newtheorem{remark}[theorem]{Remark}
\newtheorem{definition}[theorem]{Definition}

\newcommand{\innerthmname}{}% initialize

\theoremstyle{definition}

\makeatletter
\def\namedlabel#1#2{\begingroup
	#2%
	\def\@currentlabel{#2}%
	\phantomsection\label{#1}\endgroup
}

\newcommand*\owedge{\mathpalette\@owedge\relax}
\newcommand*\@owedge[1]{%
	\mathbin{%
		\ooalign{%
			$#1\m@th\bigcirc$\cr
			\hidewidth$#1\m@th\wedge$\hidewidth\cr
		}%
	}%
}
\makeatother

\newcommand{\R}{\mathbf{R}}

\newcommand{\N}{\mathbf{N}}
\newcommand{\Ss}{\mathbf{S}}
\newcommand{\Z}{\mathbb{Z}}

\DeclareMathOperator{\supp}{supp}

\definecolor{JRcolor}{RGB}{130,50,180}   % purple
\definecolor{JWcolor}{RGB}{40,130,40}    % green
\definecolor{ADlTcolor}{RGB}{200,50,50}  % red
\definecolor{JHAcolor}{RGB}{0,90,180}    % dark blue

\newif\ifdraft
\drafttrue   % set \draftfalse for final version

\ifdraft
  \usepackage{pdfcomment}
  \newcommand{\JHAblock}[1]{\pdfcomment[icon=Note,color={0.00 0.35 0.71},subject={JHA}]{#1}}
  \newcommand{\JRblock}[1]{\pdfcomment[icon=Note,color={0.51 0.20 0.71},subject={JR}]{#1}}
  \newcommand{\ADlTblock}[1]{\pdfcomment[icon=Note,color={0.78 0.20 0.20},subject={ADlT}]{#1}}
  \newcommand{\JWblock}[1]{\pdfcomment[icon=Note,color={0.16 0.51 0.16},subject={JW}]{#1}}
\else
  \newcommand{\JHAblock}[1]{}
  \newcommand{\JRblock}[1]{}
  \newcommand{\ADlTblock}[1]{}
  \newcommand{\JWblock}[1]{}
\fi

\title[Classification of singular Yamabe metrics]{Classification of fractional, singular Yamabe 
metrics on a twice-punctured sphere I}

\author[J.H. Andrade]{Jo\~{a}o Henrique Andrade}
\author[A. DelaTorre]{Azahara DelaTorre}
\author[J.M. do \'O]{Jo\~{a}o Marcos do \'O}
\author[J. Ratzkin]{Jesse Ratzkin}
\author[J. Wei]{Juncheng Wei}

\address[J.H. Andrade]{
	Institute of Mathematics and Statistics,
	University of S\~ao Paulo
	\newline\indent 
	05508-090, S\~ao Paulo-SP, Brazil
	}
\email{\href{mailto:andradejh@ime.usp.br}{andradejh@ime.usp.br}}

\address[A.DelaTorre]{
Dipartimento di Matematica Guido Castelnuovo,
Facolt\`a Scienze matematiche, fisiche e naturali,
Sapienza Universit\`a di Roma,
Piazzale Aldo Moro, 5, 00185 Roma RM.
}
\email{\href{mailto:azahara.delatorrepedraza@uniroma1.it}{azahara.delatorrepedraza@uniroma1.it}}

\address[J.M. do \'O]{Department of Mathematics,
	Federal University of Para\'{\i}ba
	\newline\indent 
	58051-900, Jo\~ao Pessoa-PB, Brazil}
\email{\href{mailto:jmbo@pq.cnpq.br}{jmbo@pq.cnpq.br}}

\address[J. Ratzkin]{Department of Mathematics,
	Universit\"{a}t W\"{u}rzburg
	\newline\indent
	97070, W\"{u}rzburg-BA, Germany}
\email{\href{mailto:jesse.ratzkin@uni-wuerzburg.de}{jesse.ratzkin@uni-wuerzburg.de}}

\address[J. Wei]{
	Department of Mathematics,
	Chinese University of Hong Kong
	\newline\indent 
	Room 220, Lady Shaw Building, Shatin, N. T., Hong Kong}
\email{\href{mailto:wei@math.cuhk.edu.hk}{wei@math.cuhk.edu.hk}}

\subjclass[2020]{35J60, 35R11, 53C21, 53A30, 35B33, 35B40}
\keywords{fractional Yamabe problem, Delaunay metrics, fractional Paneitz operator, singular solutions, conformal geometry, uniqueness and classification}

\begin{document}

%====================================================
% Abstract
%====================================================
\begin{abstract}
The Delaunay metrics form a family of conformally flat, constant fractional (order $s\in (0,1)$) 
$Q$-curvature metrics on a twice-punctured sphere. After a M\"obius transformation, they are all
rotationally symmetric and periodic, and admit several elegant variational descriptions. We prove 
that there exists $\delta > 0$ such that when $s\in (1-\delta,1)$ any complete, conformally 
flat constant $Q$-curvature metric on a 
twice-punctured sphere is a Delaunay metric. Along the way, we prove an {\it a priori}
bound for the conformal factor of these metrics which is uniform in $s$.
 \end{abstract}

\maketitle

%====================================================
% Introduction
%====================================================
\section{Introduction} 
Much of geometric analysis in the last 40 years has concentrated on 
conformally invariant equations. The first step in this program for a given equation 
is to demonstrate the existence of 
a solution in a given conformal class. As a next step, one can try to 
classify the solutions, or at least prove compactness of the solution 
set in the form of {\it a priori} bounds. In the case of the conformal 
class of the round sphere, the conformal invariance of the equation 
combines with the noncompactness of the group of conformal transformations 
to complicate both the {\it a priori} bounds and the uniqueness of 
solutions. 

In the present paper, we consider the following conformally invariant equation 
\begin{equation} \label{frac_paneitz_eqn}
\mathbb{P}_s u = c_{n,s} u^{\frac{n+2s}{n-2s}} \quad {\rm in} \quad \Omega,
\end{equation} 
where $s\in(0,\frac{n}{2})$, $\mathbb{P}_s$ is the family of fractional Paneitz operators 
constructed by Graham and Zworski in \cite{graham-zworski} (see 
also \cite{chang-gonzalez}) and $c_{n,s}>0$ is the normalization 
constant defined as
$$ c_{n,s} = 2^{2s} \left ( \frac{\Gamma \left (\frac{1}{2} \left ( 
\frac{n}{2} + s \right ) \right ) }{ \Gamma \left ( \frac{1}{2} 
\left ( \frac{n}{2} -s\right ) \right ) } \right )^2.$$
The function $\Gamma$ 
is the classical Gamma function defined as $\Gamma(z) = \int_0^\infty t^{z-1} e^{-t} dt.$
We remark for future reference that
$$\lim_{s \nearrow 1} c_{n,s} = \left ( \frac{n-2}{2} \right )^2.$$%
Here we assume $\Omega := M\setminus \Lambda$ to be the so-called regular set, where 
$(M,g)$ is an $n$-dimensional smooth, compact Riemannian manifold without boundary, 
with $n\geqslant 3$ and $\Lambda\subset M$ 
is a closed subset called the singular set. (The restriction $n \geq 3$ allows 
us to avoid singularities of the $\Gamma$ function as $s \nearrow 1$.) 
We adopt the same terminology as in the study of 
constant scalar 
curvature metrics and denote a solution $u\in \mathcal{C}^\infty(\Omega)$ of \eqref{frac_paneitz_eqn} 
as a fractional, singular Yamabe solution, and the associated metric 
$\widetilde{g} = u^{\frac{4}{n-2s}} g$ as a fractional, singular Yamabe metric. 

We are mostly interested in this equation on subdomains of the sphere, 
{\it i.e.} \(\Omega:=\mathbf{S}^n\setminus \Lambda \) with 
\(\Lambda\subset \mathbf {S}^n\) a closed set. We denote the 
usual round metric as $g_\circ$. 
In this locally conformally flat setting, $\mathbb{P}_s$ has the explicit form (originally 
given by Branson in \cite{Branson_funct_det}) given by
\begin{equation} \label{paneitz_op_sph1}
\mathbb{P}_s = \frac{\Gamma \left ( \frac{n}{2} - s \right ) \Gamma \left ( B + s 
+ \frac{1}{2} \right ) } { \Gamma \left ( \frac{n}{2} + s \right ) 
\Gamma \left ( B - s + \frac{1}{2} \right ) }
\quad {\rm with} \quad B = \sqrt { -\Delta_{g_{\circ}} + \left ( \frac{n-1}{2} 
\right )^2}. 
\end{equation} 
Later, we will see a similar expansion for $\mathbb{P}_s$ in the cylindrical 
setting, which was first derived in \cite{Azahara2018}. 

Now, let $\Lambda=\{N, S\}$ with $N, S\in \mathbf{S}^n$ denote its north and south poles, respectively.
The conformally flat metric $g=U^{\frac{4}{n-2s}} g_{\circ}$ has constant
$Q_s$-curvature if the (positive) conformal factor 
$U: \Ss^n \backslash \{N,S\} \rightarrow (0,\infty)$ satisfies the PDE 
\begin{equation} \label{const_q_sph}\tag{${\mathcal{Q}}^\circ_{n,s}$}
\mathbb{P}_s U = c_{n,s} U^{\frac{n+2s}{n-2s}} \quad {\rm on} \quad \mathbf{S}^n\setminus \{N,S\}. 
\end{equation} 
In this setting, completeness of the metric $g\in [g_{\circ}]$ is equivalent to the boundary condition 
\[
\liminf_{p \rightarrow \{N,S\}} U(p) = \infty.
\]
Equivalently, we say that $U\in\mathcal{C}^\infty(\mathbf{S}^n\setminus\{N,S\})$ is a singular solution to \eqref{const_q_sph}.

Let $g_{\rm flat}$ be the 
standard Euclidean metric on $\R^n$.
We can use conformal invariance to transfer this problem 
to (a subdomain of) Euclidean space. Let $\Pi:\Ss^n \backslash \{ N\} 
\rightarrow \R^n$ be the standard stereographic projection and let 
$u_{\circ} : \R^n \rightarrow (0,\infty)$ be the standard  spherical solution 
({\it i.e.} Aubin--Talenti-type bubble or Chen--Li--Ou bubble \cite{MR2131045}) given by
\begin{equation} \label{sph_soln1}
 \quad u_{\circ}(x) 
= \hat{c}_{n,s} \left ( \frac{1+|x|^2}{2} \right )^{\frac{2s-n}{2}} 
\end{equation} 
with
\[
\hat{c}_{n,s} = 2^{\frac{2s-n}{2}} \left ( \frac{\Gamma \left ( \frac{1}{2} 
\left ( \frac{n}{2} + s\right ) \right ) }{\Gamma \left ( \frac{1}{2}
\left ( \frac{n}{2} - s \right ) \right ) }\right )^{\frac{2s-n}{2s}} 
\left ( \frac{\Gamma \left ( \frac{n}{2} - s\right )} {\Gamma\left ( \frac{n}{2} + s
\right ) } \right )^{\frac{2s-n}{4s}}.
\]
Here we remark that $\hat{c}_{n,s}> 1$ (cf. \cite[Proposition~6.1]{Azahara2018}) 
and that
\begin{equation} \label{normalization_limit}
\lim_{s \nearrow 1} \hat{c}_{n,s} = \left ( \frac{n}{n-2} \right)^{\frac{n-2}{4}}.\end{equation} %

We can now write a conformally 
flat metric in both the spherical and Euclidean gauges, where the two are 
related by 
\begin{equation} \label{sph_to_euc_gauge}
g=U^{\frac{4}{n-2s}}g_{\circ} = u^{\frac{4}{n-2s}} g_{\rm flat} \quad {\rm with} \quad u=U \cdot u_{\circ}. 
\end{equation}
We observe that the spherical solution described above does not 
correspond to the unit sphere, but rather to a sphere $\mathbf{S}^n(\rho)$ 
of radius $\rho=\rho(n,s)>0$ given by
$$\rho_{n,s}=2^{-2} \left ( \frac{\Gamma \left ( \frac{1}{2} 
\left ( \frac{n}{2} + s\right ) \right ) }{\Gamma \left ( \frac{1}{2}
\left ( \frac{n}{2} - s \right ) \right ) }\right )^{-2/s} 
\left ( \frac{\Gamma \left ( \frac{n}{2} - s\right )} {\Gamma\left ( \frac{n}{2} + s
\right ) } \right )^{-1/s}. $$
We choose this scaling so that the sphere fits with the form of the Delaunay solutions we derive below in \S~\ref{sec:del_existence}. 

It will be useful to transfer without further comment between these two gauges, as well as the cylindrical gauge mentioned earlier. 
In the Euclidean gauge, the positive conformal factor $u: \R^n \backslash \{ 0 \} \rightarrow (0,\infty)$ satisfies
\begin{equation} \label{const_q_euc} \tag{$\overline{\mathcal{Q}}_{n,s}$}
(-\Delta)^s u = c_{n,s} u^{\frac{n+2s}{n-2s}} \quad {\rm in} \quad \mathbf{R}^n\setminus\{0\},
\end{equation} 
where $\Delta$ is the standard Euclidean Laplacian and $(-\Delta)^s$ is the 
fractional Laplacian, which can be understood as the principal part of a singular integral. 
Here, the boundary condition translates to
\(\liminf_{x \rightarrow 0} u(x) = \infty.\)

In previous work, DelaTorre, del Pino, Gonz\'alez and Wei \cite{MR3694655}
constructed a family of solutions on the twice-punctured sphere. These 
solutions are, after an appropriate conformal dilation, all rotationally symmetric. 
They generalize the well-known Delaunay metrics in the scalar curvature setting, 
which were found by Schoen \cite{MR994021}, and we call them Delaunay metrics 
as well. A Delaunay metric is uniquely determined by its fundamental period, or 
equivalently by its necksize. 
More precisely, the fractional Delaunay family can be equivalently characterized as the 
set of even, positive, periodic 
solutions \(v_\varepsilon:\mathbf{R}\to(0,\infty)\) of a nonlocal initial value problem  (see \eqref{eq_log-cyl}) that fixes the necksize parameter 
\(\varepsilon = \min_{t \in \R} v_\varepsilon(t) .\)
For each 
necksize \(\varepsilon\in(0,1]\), the Euclidean conformal factor
\begin{equation}\label{eq:euclideandelaunay}
    u_\varepsilon(x) = |x|^{\frac{2s-n}{2}}v_\varepsilon(-\ln|x|),
\end{equation}
solves \eqref{const_q_euc} and defines $g_\varepsilon=u_\varepsilon^{\frac{4}{n-2s}}g_{\rm flat}$ 
a complete constant \(Q_s\)–curvature metric on 
\(\R^n \backslash \{ 0 \} \). We can also write out the Delaunay metric in 
the spherical gauge as 
$$g_\varepsilon = U_\varepsilon^{\frac{4}{n-2s}} g_\circ \quad {\rm and} \quad u_\varepsilon = 
U_\varepsilon \cdot u_\circ.$$
Later, Jin and Xiong \cite{MR4266239} used a dual formulation of the equation to extend the 
construction of the Delaunay metrics to all values of $s \in (0,n/2)$. 

In what follows, we call $u\in \mathcal{C}^\infty(\mathbf{R}^n\setminus\{0\})$ a singular solution to \eqref{const_q_euc} if it does not extend continuously to the origin, {\it i.e.} $\lim_{x\to 0}u(x)=\infty$; otherwise, we call it  a nonsingular solution. 

Our first main result is a partial classification of singular solutions to the fractional Yamabe equation, which can be stated as
\begin{theorem} \label{main_thm} 
Let $n \in \N$ with $n \geqslant 3$. There exists $\delta > 0$ such that 
if $s \in (1-\delta, 1)$ and $u\in \mathcal{C}^\infty(\mathbf{R}^n\setminus\{0\})$ 
is a singular solution to \eqref{const_q_euc}, then there exists $\varepsilon \in 
(0, 1]$ and $R>0$ such that 
$$u = \bar{u}_{\varepsilon, R} \quad {\rm with} \quad \bar{u}_{\varepsilon, R} (x) = R^{\frac{2s-n}{2}} 
u_\varepsilon \left ( \frac{x}{R} \right ),$$
where $\bar{u}_\varepsilon\in \mathcal{C}^\infty(\mathbf{R}^n\setminus\{0\})$ is the Euclidean Delaunay solution of necksize $\varepsilon$ $($see \eqref{eq:Delaunay_eucledian}$)$. 
\end{theorem}

For a given doubleton set $\{ p, q\} \in \Ss^n$ with $p \neq q$ we define the 
marked moduli space of complete metrics with constant fractional $Q$-curvature as 
\begin{equation*}
\mathcal{M}_{\{ p,q\}} = \left \{ g=U^{\frac{4}{n-2s}}g_{\circ} \in [g_{\circ}]: 
\text{$U\in\mathcal{C}^\infty(\mathbf{S}^n\setminus \{p,q\})$ is a positive 
solution to \eqref{const_q_sph}}\right \}/\sim,
\end{equation*} 
where the quotient is taken under conformal diffeomorphisms preserving the 
singular set $\{N,S\}$ and is furnished with the Gromov-Hausdorff topology. 

In this fashion, we may reformulate our main theorem as
\begin{corollary}\label{main_cor}
Let $n \in \N$ with $n \geqslant 3$ and let $p, q \in \Ss^n$ with 
$p \neq q$. There exists 
$\delta> 0$ such that if $s\in (1-\delta,1)$ then each element of 
$\mathcal{M}_{\{p,q\}}$ is a Delaunay metric. 
\end{corollary}

We use the same technique as in the proof of Theorem \ref{main_thm} to 
prove that all Delaunay solutions are nondegenerate. 
This condition is 
easiest to define in the cylindrical gauge. 
For any $v:\R\times \Ss^{n-1} \rightarrow (0,\infty)$ solution to \eqref{eq_log-cyl}, we set the associated linearized operator as
\[
\mathbb{L}_v := P_s + c_{n,s} 
- \tfrac{n+2s}{n-2s}\,c_{n,s}\,v^{\frac{4s}{n-2s}}
\]
We say that $v$ is nondegenerate 
if $w \in L^2(\R \times \Ss^{n-1})$ and $\mathbb{L}_v(w) = 0$ implies 
$w \equiv 0$. 

In this setting, our second main result is
\begin{theorem} \label{nondegen_thm} 
Let $n \in \N$ with $n \geqslant 3$. There exists $\delta > 0$ such that 
if $s \in (1-\delta, 1)$ and $v : \R \times \Ss^{n-1} \rightarrow (0,\infty)$
satisfies \eqref{eq_log-cyl} and $v \neq v_\circ$, then $v$ is nondegenerate. 
\end{theorem} 

In the scalar curvature setting, the 
conformal factor satisfies a second order ODE, and so the uniqueness follows very 
quickly from phase-plane analysis. However, as demonstrated in \cite{Azahara2018}, one cannot 
apply these techniques directly in the nonlocal setting. Nonetheless, we can prove 
uniqueness if $s$ is close to, but less than one.

Our proof of uniqueness relies on two ideas. The first is a Cauchy uniqueness
property, which we prove in \S~\ref{sec:glob_uniq}: if $s$ is close to one, two solutions of \eqref{eq_log-cyl}
having the same value and the same derivative at one point must coincide. We prove this by contradiction, assuming
there exist a sequence $s_k \nearrow 1$ and pairs of distinct solutions $v_k$, $\widetilde{v}_k$ with the same Cauchy
data. A rescaling of the difference $v_k - \widetilde{v}_k$ converges to a solution
of a homogeneous second order ODE. This limit must be zero, which contradicts the normalization 
of our rescaling. The
passage to the limit relies on phase plane analysis of the limit equation and precise estimates 
for the integral kernel of the nonlocal operator, which in turn give us an {\it a priori}
bound which is uniform in $s$. We derive these {\it a priori} estimates in \S~\ref{sec:uniform_bound}. 
The second idea, carried out in \S~\ref{sec:reflection}, is to apply this property to a solution
and its reflection about a critical point, as in \cite[Lemma~8]{FrankKonig}. It follows that each 
solution is symmetric about each of its critical points,
hence periodic, and that it is determined by its necksize.
We expect this strategy to be robust and to have applications in other 
fractional geometric problems, and indeed we adapt our technique to 
prove Theorem \ref{nondegen_thm} in \S~\ref{sec:nondegen}. 
We conclude this paper with a brief outlook in \S~\ref{sec:outlook}. First, 
in \S~\ref{sec:extension}, we discuss work in progress \cite{our_followup} in 
which we extend the classification to $s \in (1,1+\delta)$, where once again $\delta>0$ 
is sufficiently small. Next, in \S~\ref{sec:applications}, we describe some possible 
applications of our classification result, including refined 
asymptotics of solutions with isolated singularities and a general compactness 
result. 

\begin{remark} After we circulated an earlier version of this 
paper, Jin, Liao, and Xiong \cite{Jin_Liao_Xiong} completed the classification 
for all $s \in (0,1)$ using different methods based on the Caffarelli-Silvestre 
extension.
\end{remark} 

%%====================================================
% Section: Qualitative properties of Delaunay solutions
%====================================================
\section{Qualitative properties of Delaunay solutions} 
In this section, we first describe the 
reduction to a one-dimensional problem. Next, we describe the 
Delaunay solutions. 

\subsection{Rotational symmetry} \label{sec:rot_symm}
Here, we describe why all the complete, constant $Q_s$-curvature metrics on a twice-punctured 
sphere depend only on one variable. 
We begin with a complete, conformally flat, constant $Q_s$-curvature 
metric $g = U^{\frac{4}{n-2s}} g_{\circ}$ on a twice-punctured sphere. After a 
conformal dilation, we can assume the two punctures are antipodal, and after a 
global rotation, we can take the punctures to be the north and south poles, which 
we denote respectively as $N$ and $S$. 
We now use stereographic projection to transfer the problem to 
$\R^n \backslash \{ 0 \}$. Following this change of coordinates, we arrive at \eqref{const_q_euc}.

One can now apply \cite{MR3366748} to see that $u$ is rotationally invariant. Alternatively, 
one can argue as follows. Ng\^{o} and Ye \cite{MR4438901}
prove that, because the singular set $\{ 0 \}$ has capacity $0$, the 
function $u$ is $s'$-superharmonic for each $0 < s' < s$. One can 
now use the method of moving planes to show that $u$ is radial. 
One can also use the method of moving planes in integral form (see
\cite{MR3626036} and \cite{MR2200258}) to show the radial symmetry directly. 

\subsection{Existence of Delaunay solutions} \label{sec:del_existence} 
The Delaunay solutions form a particularly nice family of rotationally 
symmetric solutions of \eqref{const_q_euc} on $\R^n \backslash \{ 0 \}$, 
or, equivalently, of \eqref{const_q_sph} on $\Ss^n \backslash \{ N,S\}$. In 
the classical setting ({\it i.e.} $s=1$) the existence of the Delaunay solutions 
goes back to the 1930s, due to the work of Fowler \cite{fowler}. 
In the fractional setting, DelaTorre, Del Pino, Gonz\'alez and 
Wei \cite{MR3694655} first established their existence provided 
$s\in(0,1)$. Later, Jin and Xiong \cite{MR4266239} extended this 
construction to the range of exponents $s\in (1,\tfrac{n}{2})$. 

The Emden-Fowler change of coordinates lies at the heart of 
this construction. Given $u:\R^n \backslash \{ 0 \} 
\rightarrow \R$, we let 
\begin{equation} \label{emden_fowler_change}
v (t,\theta) = \mathfrak{F} (u)(t,\theta) = e^{\left ( 
\frac{2s-n}{2}\right )t} u(e^{-t} \theta) \quad {\rm with} \quad t= 
-\ln |x| \quad {\rm and} \quad \theta = \frac{x}{|x|}, 
\end{equation} 
obtaining a new function $v:\R \times \Ss^{n-1} \rightarrow \R$. 
One can, of course, invert this change of coordinates, giving us  
$$u(x) = \mathfrak{F}^{-1}(v)(x) = |x|^{\frac{2s-n}{2}} v\left ( -\ln |x|, 
\frac{x}{|x|} \right ).$$
For later reference, we observe that the standard spherical 
solution $u_\circ$ transforms to 
\begin{equation} \label{bubble_in_fowler_coords}
v_\circ (t) = \mathfrak{F}(u_\circ) (|x|) = \widehat{c}_{n,s} 
(\cosh t)^{\frac{2s-n}{2}}. 
\end{equation} 

The construction in \cite{MR3694655} starts with the Ansatz that 
the solution $u$ has the form 
$$u(x) = |x|^{\frac{2s-n}{2}} \widetilde{v} (|x|),$$
where $\widetilde{v} : (0,\infty) \rightarrow (0,\infty)$ 
is a positive smooth function. In fact, one can use the {\it a priori} bounds of 
\cite{MR3366748} to show that $\widetilde{v}$ is bounded, and so the 
integral operators written below are convergent. 
This change of variables transforms \eqref{const_q_euc} into 
\begin{equation} \label{radial1} 
\kappa_{n,s}\int_0^\infty \int_{\Ss^{n-1}} \frac{(r^{\frac{2s-n}{2}} \widetilde{v}(r)
- \rho^{\frac{2s-n}{2}} \widetilde v(\rho))\rho^{n-1} }
{|r^2+\rho^2 - 2r \rho \langle \theta,\varphi\rangle |^{\frac{n+2s}{2}}} d\varphi d\rho 
=c_{n,s} u(x)^{\frac{n+2s}{n-2s}} 
\end{equation} 
with
$$\kappa_{n,s} = \pi^{-n/2} 2^{2s} \frac{\Gamma\left ( \frac{n}{2}+s\right )}
{\Gamma(1-s)} s.$$

They then make the change of variables $\overline{\rho} = \rho/r$. Using 
the identity 
$$\widetilde{v} (r) = \left ( 1-\overline{\rho}^{\frac{2s-n}{2}} \right ) 
\widetilde {v}(r) + \overline{\rho}^{\frac{2s-n}{2}} \widetilde{v}(r),$$
the equation \eqref{radial1} becomes 
\begin{equation} \label{radial2} 
\kappa_{n,s} \int_0^\infty \int_{\Ss^{n-1}} \frac{\overline{\rho}^{n-1-\frac{n-2s}{2}}
(\widetilde {v}(r) - \widetilde{v}(\overline{\rho} r))}
{|1+\overline{\rho}^2 - 2\overline{\rho} \langle \theta, \varphi
\rangle |^{\frac{n+2s}{2}}} d\overline{\rho} d\varphi+ A \widetilde{v}(r) = c_{n,s} \widetilde {v}
(r)^{\frac{n+2s}{n-2s}},
\end{equation} 
where 
$$A = \int_0^\infty \int_{\Ss^{n-1}} \frac{(1-\overline{\rho}^{\frac{2s-n}{2}})
\overline{\rho}^{n-1}} {(1+\overline{\rho}^2 - 2\overline{\rho} \langle \theta, \varphi\rangle 
)^{\frac{n+2s}{2}}} d\varphi d\overline{\rho}$$ 
is a positive constant; in fact, our normalizations imply $A=c_{n,s}$ (see \cite[Proposition 2.7]{Azahara2018}). 
Finally, we complete the Emden-Fowler 
change of variables by letting $t=-\ln r$, $\tau=-\ln \rho$ and 
$v(t) = \widetilde{v} (e^{-t})$, 
which transforms \eqref{radial2} into 
\begin{equation}\label{eq_log-cyl}\tag{$\widehat{\mathcal{Q}}_{n,s}$}
    P_s(v)+c_{n,s}v= c_{n,s} v^{\frac{n+2s}{n-2s}} \quad {\rm in} \quad \mathbf{R},
\end{equation}
where 
\begin{equation} \label{cyl_GJMS1} 
P_s (v) (t) = \int_{-\infty}^\infty 
(v(t) - v(\tau))K_s(t-\tau) d\tau
\end{equation} 
and
\begin{equation*}
K_s(\xi) = \kappa_{n,s}\, 2^{-\frac{n+2s}{2}}
\int_{\Ss^{n-1}} \frac{d\varphi}{(\cosh \xi - \langle \theta, \varphi \rangle)^{\frac{n+2s}{2}}} .
\end{equation*}
Indeed, with $\overline{\rho} = e^{-\xi}$, one has
\[
\overline{\rho}^{\,n-1-\frac{n-2s}{2}}\, d\overline{\rho} = - e^{-\frac{n+2s}{2}\xi}\, d\xi
\quad {\rm and} \quad
1 + \overline{\rho}^2 - 2 \overline{\rho} \langle \theta, \varphi \rangle = 2 e^{-\xi} (\cosh \xi - \langle \theta, \varphi \rangle).
\]
By rotation invariance, $K_s$ does not depend on $\theta$, and it is positive and even.

From this formulation, one can search for periodic solutions with minimal 
period $L$ by setting 
\begin{equation}\label{periodickernal}
    K_{s,L} (t-\tau) = \sum_{j\in \mathbb{Z}} K_s(t-\tau-jL)
\end{equation}
and minimizing 
\begin{equation} \label{del_min}
\mathcal{E}_{s,L}(v) = \frac{\frac{1}{2} \int_0^L\int_0^L (v(t) - v(\tau))^2 K_{s,L}(t-\tau)dt d\tau 
+ c_{n,s}\int_0^L v(t)^2 dt}
{\left ( \int_0^L v(t)^{\frac{2n}{n-2s}}dt \right )^{\frac{n-2s}{n}}} ,
\end{equation} 
where one minimizes over all the functions such that the numerator is 
finite. The main theorem of \cite[Theorem~1.1]{MR3694655} states that for each 
$L>0$ the functional $\mathcal{E}_{s,L}$ in \eqref{del_min} has a nonzero 
minimizer. Furthermore, there exists $L_0$ such that when 
$0<L <L_0$ this minimizer is constant and when $L>L_0$ the minimizer is 
nonconstant. 

We complete this description by remarking that the preceding paragraphs 
classify the Delaunay solutions in the cylindrical gauge. We can subsequently 
change coordinates to write the Delaunay solution in the Euclidean or the 
spherical gauges. In the Euclidean gauge, the solution reads
\begin{equation}\label{eq:Delaunay_eucledian}
\bar{u}_\varepsilon (x) = \mathfrak{F}^{-1} (v_\varepsilon) (x) = 
|x|^{\frac{2s-n}{2}} \bar{v}_\varepsilon (-\ln |x|).     
\end{equation}
One can then transfer this solution to the sphere using stereographic 
projection, obtaining 
$${U}_\varepsilon  = \bar{u}_\varepsilon \circ \Pi= \mathfrak{F}^{-1}\circ 
\bar{v}_\varepsilon \circ \Pi$$

%====================================================
% Section: Sharp estimates
%====================================================
With these preliminaries in hand, we now turn to the {\it a priori} estimates.

\section{A priori estimates}
In this section, we establish two upper bounds for positive solutions to the critical
fractional equation on the punctured sphere and its cylindrical counterpart. We begin with a
general upper bound in the setting that the singular set is a closed subset of $\Ss^n$ with
zero capacity. The constant in this estimate depends on $s$ in a way we do not control, so afterwards we
specialize to the case of a twice-punctured sphere, where we prove a bound which is uniform as $s \nearrow 1$.

\subsection{Upper bound estimate} \label{sec:general_upper_bound}
In this section, we derive {\it a priori} upper bounds in a more general setting. 

\begin{lemma}\label{lm:sharpestimate}
    Let $n \in \N$ and $s\in(0,1)$ with $n \geqslant 3$, and let $\Lambda\subset \Ss^n$ be a closed set with zero capacity.
    If  
    $U:\Ss^n \backslash \Lambda \rightarrow (0,\infty)$ satisfies \eqref{const_q_sph}, then there exists $C>0$ depending only on $n$ and $s$ such that  
    \begin{equation} \label{general_upper_bound} 
        U(x) \leqslant C\operatorname{dist}_{g_{\circ}}(x,\Lambda)^{\frac{2s-n}{2}}.
    \end{equation} 
\end{lemma}

\begin{remark} 
In the case that $\Lambda$ is finite, this result follows 
from \cite[Proposition 3.1]{MR3366748}. This more general proof when $\Lambda$ is 
a closed set with zero capacity follows the model first used by 
Pollack in \cite[Theorem 4.3]{MR1266101} and has since been adapted in a number of other 
places, such as \cite{Marques_iso_blowup}, \cite{MR2219215}, \cite{chang-han-yang}, and 
\cite{cmpt_6th_order}. Since this same argument has appeared now a number of times 
in the literature, we only provide its sketch. 
\end{remark} 

\begin{proof}
We argue by contradiction, adapting the blow-up argument of Pollack~\cite[Theorem 4.3]{MR1266101}, 
combined with the convexity argument of Schoen~\cite{MR1173050} and its conformal extension by 
Qing and Raske~\cite[Theorem 4.1]{MR2219215}.

Fix a regular point $x_\infty \in \Ss^n \setminus \Lambda$ and consider the stereographic projection 
with pole at $x_\infty$, identifying $\Ss^n \setminus \{x_\infty\}$ with $\R^n$. Let 
$u:\R^n \setminus \widehat \Lambda \rightarrow (0,\infty)$ be the stereographic projection of $U$, 
where $\widehat \Lambda$ is the image of $\Lambda$ under the projection. Then $u$ satisfies the 
Euclidean conformally invariant equation
\[
(-\Delta)^s u = c_{n,s} u^{\frac{n+2s}{n-2s}} \quad {\rm in} \quad \R^n \setminus \widehat \Lambda.
\]
Since the singular set $\widehat \Lambda$ has capacity 
zero, the results of~\cite{MR4438901} imply that $u$ is $s'$-superharmonic for every $s'\in(0,s)$. 

Next, we define the conformal metric $g_u = u^{\frac{4}{n-2s}} g_{\rm flat}$. By applying the moving 
planes method of Schoen~\cite{MR1173050} adapted to the fractional setting by 
Qing and Raske~\cite[Theorem 4.1]{MR2219215}, we deduce that every Euclidean ball in the 
domain $\R^n \setminus \widehat \Lambda$ is strictly mean-convex with respect to $g_u$. This 
key convexity property will be used to obtain a contradiction.

Now, we suppose by contradiction that the desired estimate~\eqref{general_upper_bound} fails. 
Then, there exists a sequence of data $\Lambda_k$, solutions $u_k$, points $x_{0,k}$, and 
radii $\rho_k>0$ such that $B_{\rho_k}(x_{0,k}) \subset \R^n \setminus \widehat \Lambda_k$ and 
the function
\[
f_k(x) := \left(\rho_k - |x - x_{0,k}|\right)^{\frac{2s-n}{2}} u_k(x)
\]
attains arbitrarily large supremum in $B_{\rho_k}(x_{0,k})$. That is, there exists 
$x_{1,k} \in B_{\rho_k}(x_{0,k})$ such that
\[
M_k := f_k(x_{1,k}) = \sup_{x \in B_{\rho_k}(x_{0,k})} f_k(x) \to \infty \quad {\rm as} \quad k \to \infty.
\]
We define the scaling factor $S_k := \frac{1}{2} M_k^{-\frac{2}{n-2s}}$, and consider the blow-up sequence
\[
w_k(y) := S_k^{\frac{n-2s}{2}} u_k(x_{1,k} + S_k y),
\]
which, by construction, satisfies $w_k(0) = {2^{\frac{2s-n}{2}}}$ and
\[
(-\Delta)^s w_k = c_{n,s} w_k^{\frac{n+2s}{n-2s}} \quad {\rm in} \quad
\Omega_k := \frac{1}{S_k} \left( \R^n \setminus \widehat \Lambda_k - x_{1,k} \right).
\]

Because the singular sets $\widehat \Lambda_k\subset\mathbf{R}^n$ diverge to infinity in the rescaled variables, and 
thanks to the uniform {\it a priori} Harnack inequality from~\cite[Theorem 1.2]{TanXiong2011}, we obtain 
local $\mathcal{C}^{2,\alpha}$ estimates for $w_k$ on compact subsets. Passing to a subsequence, we conclude 
that $w_k \to w_\infty$ in $\mathcal{C}^2_{\mathrm{loc}}(\R^n)$, where $w_\infty \in \mathcal{C}^2(\mathbf{R})$ satisfies $w_\infty(0) = {2^{\frac{2s-n}{2}}}$ and solves 
\[
(-\Delta)^s w_\infty = c_{n,s} w_\infty^{\frac{n+2s}{n-2s}} \quad {\rm in} \quad \R^n.
\]
By the classification result of 
Chen, Li and Ou \cite[Theorem 1.1]{MR2200258}
$w_\infty$ must be the standard spherical solution
\[
w_\infty(y) = \left( \frac{\varepsilon}{1 + \varepsilon^2 |y - y_0|^2} \right)^{\frac{n-2s}{2}}.
\]
for some $\varepsilon > 0$ and $y_0 \in \R^n$.

Since $w_k \to w_\infty$ locally uniformly, we find that for large $k\gg1$, the metrics $g_{u_k}$ 
become nearly spherical on large balls centered at $x_{1,k}$. In particular, there exists $R>0$ such 
that the ball $B_{R}(x_{1,k})$ has a boundary that is strictly mean-concave with respect to the 
conformal metric $g_{u_k}$. But this contradicts the earlier geometric result that all Euclidean 
balls must have strictly mean-convex boundaries for $g_{u_k}$, as deduced from the fractional 
moving planes method.

Finally, this contradiction proves the desired {\it a priori} upper bound~\eqref{general_upper_bound}. 
The constant $C>0$ depends only on $n$ and $s$, completing the proof.
\end{proof}

\subsection{A uniform bound on the cylinder} \label{sec:uniform_bound}
In the special case where $\Lambda=\{N,S\}$ is a pair of points, we can
transform the problem to the cylinder using \eqref{emden_fowler_change},
and so \eqref{general_upper_bound} reads $v(t) \leqslant C$, where $C$ is a
positive constant depending on $n$ and $s$. Because the proof of
\eqref{general_upper_bound} argues by contradiction, it yields no control on this
constant as $s \nearrow 1$. In this subsection, we provide a direct argument
producing a bound independent of $s \in [\frac{3}{4},1)$. It uses only the structure of the
kernel $K_s$ and the radial monotonicity of singular solutions proved in \cite{MR3366748}.
We proceed in three steps. First, we derive the estimates for the kernel that we use
throughout the rest of the paper. Second, we describe the action of $P_s$ on test functions.
Third, we introduce the class of solutions with which we work and prove the uniform bound.

We start by fixing a convention. Throughout the rest of the paper, for $v \in \mathcal{C}^2(\R)$ satisfying
$|v(t)| \leqslant C e^{\frac{n}{2}|t|}$, we understand \eqref{cyl_GJMS1} in the
principal value sense. Since $K_s$ is even, this means
\begin{equation} \label{cyl_GJMS_sym}
P_s(v)(t) = \frac{1}{2} \int_{-\infty}^\infty \big( 2v(t) - v(t+\rho) - v(t-\rho) \big) K_s(\rho)\, d\rho ,
\end{equation}
and the integral converges absolutely by the next lemma, which collects the properties of the kernel.

The estimates in the following lemma sharpen those in Lemma 2.5 and 2.6 \cite{MR3694655}.
Our analysis below requires some 
more refined information, and so, for instance, we obtain higher order terms in the 
expansion of the kernel $K_s$, which is the content of \eqref{eq:2nd_moment} below.

In what follows, we set $\varsigma_{n}: (0,1)\to \mathbf{R}$ given by
\begin{equation} \label{eq:coeff_asympt}
\varsigma_{n}(s) := \pi^{-\frac{1}{2}}\, 2^{2s}\, \frac{\Gamma \left ( \frac{1}{2}+s \right )}{\Gamma(1-s)}\, s .
\end{equation}
\begin{lemma} \label{lem:kernel}
Let $n \in \N$ with $n \geqslant 3$.
There holds that $K_s$ is positive and even, and there exists $C = C(n)>0$ such that for every $s \in [\frac{3}{4},1)$ one has:
\begin{itemize}
\item[{\rm (i)}] For $0<|\rho| \leqslant 1$, one has
\[
K_s(\rho) = \varsigma_{n,s} |\rho|^{-(1+2s)} + r_s(\rho) \quad {\rm with} \quad
|r_s(\rho)| \leqslant C(1-s) \left ( 1 + |\rho|^{1-2s} \right ).
\]
\item[{\rm (ii)}] For $|\rho| \geqslant 1$, one has
\[
K_s(\rho) \leqslant C(1-s)\, e^{-\frac{n+2s}{2}|\rho|}.
\]
\item[{\rm (iii)}] For $s \nearrow 1$, one has 
\[
\lim_{s\to 1^-}\frac{\varsigma_{n,s}}{1-s} = 2.
\]
\end{itemize}
In particular, for each $\varepsilon \in (0,1]$,
\begin{equation} \label{eq:2nd_moment_a}
m_s(\varepsilon) := \int_{|\rho|<\varepsilon} \rho^2 K_s(\rho)\, d\rho \longrightarrow 2
\quad {\rm as} \quad s \nearrow 1.
\end{equation} 
and
\begin{equation} \label{eq:2nd_moment}
\int_{|\rho| \geqslant \varepsilon} \left ( 1 + e^{\frac{n}{2}|\rho|} \right ) K_s(\rho)\, d\rho \longrightarrow 0
\quad {\rm as} \quad s \nearrow 1.
\end{equation} 
\end{lemma}

\begin{proof}
Writing $q = \frac{n+2s}{2}$ and using polar coordinates on $\Ss^{n-1}$ centered at $\theta$, we obtain
\[
K_s(\rho) = \kappa_{n,s}\, 2^{-q}\, |\Ss^{n-2}| \int_0^\pi (\cosh \rho - \cos \alpha)^{-q} (\sin \alpha)^{n-2}\, d\alpha ,
\]
which is positive, even and independent of $\theta$. Since $\Gamma$ is bounded from below on $[1,2]$, it follows that
\[
\frac{1}{\Gamma(1-s)} = \frac{1-s}{\Gamma(2-s)} \leqslant C(1-s) \quad {\rm and \ so} \quad \kappa_{n,s} \leqslant C(1-s).
\]

To prove (ii), we observe that for $|\rho| \geqslant 1$ and $\alpha \in [0,\pi]$,
\[
\cosh \rho - \cos \alpha \geqslant \cosh \rho - 1 = \tfrac{1}{2}\, e^{|\rho|} \left ( 1-e^{-|\rho|} \right )^2
\geqslant \tfrac{1}{2} \left ( 1-e^{-1} \right )^2 e^{|\rho|} ,
\]
from which we deduce
\[
K_s(\rho) \leqslant \kappa_{n,s}\, |\Ss^{n-1}| \left ( 1-e^{-1} \right )^{-2q} e^{-q|\rho|} .
\]

To prove (i), let $0<|\rho| \leqslant 1$ and $\alpha \in [0,\pi]$. A Taylor expansion yields
\[
\cosh \rho - \cos \alpha = \tfrac{1}{2} (\rho^2 + \alpha^2)(1+E), \qquad |E| \leqslant C(\rho^2 + \alpha^2), \qquad 1+E \geqslant C^{-1},
\]
and
\[
(\sin \alpha)^{n-2} = \alpha^{n-2} \left ( 1 + \mathcal{O}(\alpha^2) \right ) ,
\]
which implies
\[
(\cosh \rho - \cos \alpha)^{-q} (\sin \alpha)^{n-2} = 2^q (\rho^2+\alpha^2)^{-q} \alpha^{n-2} \left ( 1 + \mathcal{O}(\rho^2+\alpha^2) \right ).
\]
The change of variables $\alpha = |\rho| x$ leads to
\[
\int_0^\pi (\rho^2+\alpha^2)^{-q} \alpha^{n-2}\, d\alpha = |\rho|^{-(1+2s)} \left ( \int_0^\infty \frac{x^{n-2}}{(1+x^2)^{q}}\, dx
- \int_{\pi/|\rho|}^\infty \frac{x^{n-2}}{(1+x^2)^{q}}\, dx \right ),
\]
where
\[
\int_0^\infty \frac{x^{n-2}}{(1+x^2)^{q}}\, dx = \frac{\Gamma \left ( \frac{n-1}{2} \right ) \Gamma \left ( \frac{1}{2}+s \right )}{2\,\Gamma \left ( \frac{n}{2}+s \right )}
\quad {\rm and} \quad
\int_{\pi/|\rho|}^\infty \frac{x^{n-2}}{(1+x^2)^{q}}\, dx \leqslant \int_{\pi/|\rho|}^\infty x^{-2-2s}\, dx \leqslant |\rho|^{1+2s} .
\]
Similarly, since the integrand decays like $x^{-2s}$ and $s \geqslant \frac{3}{4}$, it holds
\[
\int_0^\pi (\rho^2+\alpha^2)^{1-q} \alpha^{n-2}\, d\alpha \leqslant |\rho|^{1-2s} \int_0^\infty \frac{x^{n-2}}{(1+x^2)^{q-1}}\, dx
\leqslant C |\rho|^{1-2s},
\]
Also, using that the volume of the unit sphere is
\[
|\Ss^{n-2}| = \frac{2\pi^{\frac{n-1}{2}}}{\Gamma \left ( \frac{n-1}{2} \right )},
\]
the leading coefficient is
\[
\kappa_{n,s}\, |\Ss^{n-2}|\, \frac{\Gamma \left ( \frac{n-1}{2} \right ) \Gamma \left ( \frac{1}{2}+s \right )}{2\,\Gamma \left ( \frac{n}{2}+s \right )} = \varsigma_{n,s} ,
\]
and the remaining terms are bounded by
\[
C \kappa_{n,s} \left ( 1 + |\rho|^{1-2s} \right ) \leqslant C(1-s) \left ( 1+|\rho|^{1-2s} \right ).
\]

Finally, to prove (iii), we compute
\[
\frac{\varsigma_{n,s}}{1-s} = \pi^{-\frac{1}{2}}\, 2^{2s}\, s\, \frac{\Gamma \left ( \frac{1}{2}+s \right )}{\Gamma(2-s)} \longrightarrow
4\, \pi^{-\frac{1}{2}}\, \Gamma \left ( \tfrac{3}{2} \right ) = 2 \quad {\rm as} \quad s \nearrow 1.
\]
By (i), we find
\[
m_s(\varepsilon) = \frac{\varsigma_{n,s}}{1-s}\, \varepsilon^{2-2s} + \mathcal{O}(1-s) \longrightarrow 2 \quad {\rm as} \quad s \nearrow 1,
\]
and the second limit in (iii) follows from (i) on $\varepsilon \leqslant |\rho| \leqslant 1$ and from (ii) on $|\rho| \geqslant 1$,
since $\frac{n}{2} < \frac{n+2s}{2}$.
\end{proof}

\begin{remark}
In particular, the leading order term in the singular part of $K_s$ is the one-dimensional 
constant $\varsigma_{n,s} \sim 2(1-s)$, which is independent of $n$. The second moment 
of the singular part
converges to $2$, which forces the coefficient of the second derivative in the limit 
equations below to be exactly one.
\end{remark}

With the kernel estimates in hand, we can describe how $P_s$ acts on test functions as $s \nearrow 1$. 
This is the only form in which we pass to the limit in the nonlocal operator.
Graham and Zworski \cite{graham-zworski} proved that $P_s$ is self-adjoint,
but we include a proof below for the reader's convenience, which is specific to our setting. 
\begin{lemma} \label{lem:test}
Let $n \in \N$ with $n \geqslant 3$ and let $\eta \in \mathcal{C}^\infty_c(\R)$. Then, there holds
\begin{equation}\label{eq:uniform_bound}
    \sup_{s \in [\frac{3}{4},1)} \int_{-\infty}^\infty e^{\frac{n}{2}|t|}\, |P_s(\eta)(t)|\, dt < \infty
\end{equation}
and
\begin{equation}\label{eq:convergence}
\int_{-\infty}^\infty e^{\frac{n}{2}|t|}\, |P_s(\eta)(t) + \eta''(t)|\, dt \longrightarrow 0 \quad {\rm as} \quad s \nearrow 1.
\end{equation}
Moreover, if $v \in \mathcal{C}^2(\R)$ satisfies $|v(t)| \leqslant C e^{\frac{n}{2}|t|}$,then
\begin{equation}\label{eq:self_adjointness}
\int_{-\infty}^\infty \eta\, P_s(v)\, dt = \int_{-\infty}^\infty v\, P_s(\eta)\, dt \quad {\rm for} \quad s \in [\frac{3}{4},1).
\end{equation}
In particular, one has
\begin{equation}\label{eq:limit}
\int_{-\infty}^\infty P_s(\eta)\, dt = 0 \quad {\rm for} \quad s \in [\frac{3}{4},1).
\end{equation}
\end{lemma}

\begin{proof}
First, let $\supp \eta \subset [-M,M]$ and fix $\varepsilon \in (0,1]$. By Taylor's theorem, it follows
\[
\left | 2\eta(t) - \eta(t+\rho) - \eta(t-\rho) + \eta''(t)\rho^2 \right | \leqslant \tfrac{1}{3}\, \|\eta'''\|_{L^\infty(\R)}\, |\rho|^3 ,
\]
so by \eqref{cyl_GJMS_sym}, for every $t \in \R$, one has
\begin{equation} \label{eq:test_est1}
\left | P_s(\eta)(t) + \tfrac{1}{2} m_s(\varepsilon)\, \eta''(t) \right | \leqslant \tfrac{1}{6} \|\eta'''\|_{L^\infty(\R)}\, \varepsilon\, m_s(\varepsilon)
+ 2 \|\eta\|_{L^\infty(\R)} \int_{|\rho| \geqslant \varepsilon} K_s(\rho)\, d\rho,
\end{equation}
where $m_s(\varepsilon)$ is given by \eqref{eq:2nd_moment_a}. 
On the other hand, if $|t| \geqslant M+1$, then $\eta$ vanishes near $t$ and, by Lemma~\ref{lem:kernel}(ii), we find
\begin{equation} \label{eq:test_est2}
|P_s(\eta)(t)| = \left | \int_{-M}^M \eta(\tau) K_s(t-\tau)\, d\tau \right | \leqslant C(1-s)\, \|\eta\|_{L^1(\R)}\, e^{-\frac{n+2s}{2}(|t|-M)} .
\end{equation}
Taking $\varepsilon = 1$ in \eqref{eq:test_est1} and combining with \eqref{eq:test_est2} and Lemma~\ref{lem:kernel}(iii), we obtain \eqref{eq:uniform_bound}.

Second, we integrate \eqref{eq:test_est1} over $|t| \leqslant M+1$ and \eqref{eq:test_est2} over $|t| \geqslant M+1$. Letting $s \nearrow 1$ and using
Lemma~\ref{lem:kernel}(iii), one has
\[
\limsup_{s \to 1^-} \int_{-\infty}^\infty e^{\frac{n}{2}|t|}\, |P_s(\eta)(t)|\, dt \leqslant C \varepsilon
\]
where $C>0$ depends on $\eta$ but not
on $\varepsilon$. 
Since $\varepsilon \in (0,1]$ is arbitrary, \eqref{eq:convergence} follows.

Third, notice that 
\begin{align}
\int_{-\infty}^\infty\int_{-\infty}^\infty \eta\, P_s(v)\, d\tau\, dt &= \int_{-\infty}^\infty \int_{-\infty}^\infty v\, P_s(\eta)\, d\tau\, dt\\
&=\frac{1}{2} \int_{-\infty}^\infty \int_{-\infty}^\infty (v(t) - v(\tau))(\eta(t) - \eta(\tau)) K_s(t-\tau)\, d\tau\, dt ,
\end{align}
which, by Lemma~\ref{lem:kernel}, converges absolutely, since
\[
(v(t) - v(\tau))(\eta(t) - \eta(\tau)) K_s(t-\tau)=\mathcal{O}(|t-\tau|^2 K_s(t-\tau)) \quad {\rm as} \quad |t-\tau|\to 0
\]
and $\frac{n}{2} < \frac{n+2s}{2}$. Finally, taking $v \equiv 1$ yields \eqref{eq:self_adjointness}.
\end{proof}

We now single out the class of solutions with which we work. Its defining property is a bound on 
the logarithmic derivative,
which plays the role of a Harnack inequality with a constant independent of $s$.

\begin{definition} \label{def:class}
For $s \in (0,1)$, we let $\mathcal{A}_s$ be the set of positive functions $v \in \mathcal{C}^2(\R)$ which solve
\eqref{eq_log-cyl} and satisfy
\begin{equation} \label{eq:gradient_bound}
|v'(t)| \leqslant \frac{n-2s}{2}\, v(t) \quad {\rm for \ all} \quad t \in \R .
\end{equation}
\end{definition}

Observe that the class $\mathcal{A}_s$ is invariant under the translations $v \mapsto v(\cdot + T)$. 
Since $K_s$ is even, it is also invariant under the
reflection $v \mapsto v(-\,\cdot\,)$. The spherical solution $v_\circ$ belongs to $\mathcal{A}_s$, by direct computation.
Also, notice that one compare the inequality \eqref{eq:gradient_bound} to the one appearing in 
 \cite[Lemma 12]{FrankKonig}. 

The next lemma shows that this class contains all the solutions in which we are interested.
Jin, Li and Xiong \cite[Lemma 4.10]{Jin_Li_Xiong} proved the following removable singularities result.
Nevertheless, we 
include the proof for the reader's convenience. 
\begin{lemma} \label{lem:gradient}
Let $n \in \N$ and $s\in(0,1)$ with $n \geqslant 3$. If $u \in \mathcal{C}^\infty(\R^n \backslash \{0\})$ is a singular solution to
\eqref{const_q_euc} and $v = \mathfrak{F}(u)$, then $v \in \mathcal{A}_s$ and $\inf_{t \in \R} v(t) > 0$.
The same conclusion holds for every positive periodic solution $v$ to \eqref{eq_log-cyl}.
\end{lemma}

\begin{proof}
Setting $\widetilde u = c_{n,s}^{\frac{n-2s}{4s}}\,u$, one has $(-\Delta)^s \widetilde u = \widetilde u^{\frac{n+2s}{n-2s}}$ in $\R^n \setminus \{0\}$,
which is the equation studied in \cite{MR3366748}. Let us set the Kelvin transform of $u$ as
\[
u^*(x) = |x|^{2s-n}\, u \left ( \frac{x}{|x|^2} \right).
\]
which is again a positive solution to \eqref{const_q_euc}.
The origin is a nonremovable singularity of $u$ by hypothesis. It is also a nonremovable singularity of $u^*$. Otherwise, $u^*$
would extend to a positive solution on all of $\R^n$, hence would be a spherical solution by \cite{MR2200258}, and then
$u = (u^*)^*$ would be bounded near the origin. By \cite[Theorem~1.3]{MR3366748}, both $u$ and $u^*$ are radially symmetric
and strictly decreasing in $r = |x|$. Writing
\[
u(x) = r^{\frac{2s-n}{2}}\, v(-\ln r) \quad {\rm and} \quad u^*(x) = r^{\frac{2s-n}{2}}\, v(\ln r),
\]
these two monotonicity statements read
\[
-\frac{n-2s}{2} - \frac{v'(t)}{v(t)} < 0 \quad {\rm and} \quad -\frac{n-2s}{2} + \frac{v'(t)}{v(t)} < 0 \quad {\rm for \ all} \quad t \in \R,
\]
which is \eqref{eq:gradient_bound}, where we recall that $t = -\ln r$. 
Finally, applying \cite[Theorem~1.1]{MR3366748} to $u$ and to $u^*$, we obtain
$\liminf_{t \to \pm\infty} v(t) > 0$, and so $\inf_{t \in \R} v(t)>0$. If $v$ is a positive periodic solution to \eqref{eq_log-cyl}, then
$\mathfrak{F}^{-1}(v)$ and its Kelvin transform are unbounded near the origin, and the same argument applies.
\end{proof}

We can now state and prove the main result of this subsection. We emphasize that the key difference between the next proposition and Lemma~\ref{lm:sharpestimate}
is that the a priori upper bound is independent of $s$.

\begin{proposition} \label{prop:uniform_bound}
Let $n \in \N$ with $n \geqslant 3$. There exists $C_* = C_*(n)>0$ such that
\[
\sup_{t \in \R} v(t) \leqslant C_* \quad {\rm for \ all} \quad v \in \mathcal{A}_s \quad {\rm and} \quad s \in [\tfrac{3}{4},1).
\]
\end{proposition}

\begin{proof}
Fix a nonnegative $\eta \in \mathcal{C}^\infty_c((-1,1))$ with $\int_{-1}^{1} \eta\, dt = 1$, and let $t_0 \in \R$. By translation invariance, we
may assume $t_0 = 0$. Integrating \eqref{eq:gradient_bound}, we obtain
\begin{equation} \label{eq:harnack}
e^{-\frac{n-2s}{2}|t|}\, v(0) \leqslant v(t) \leqslant e^{\frac{n-2s}{2}|t|}\, v(0) \quad {\rm for \ all} \quad t \in \R .
\end{equation}
In particular, we have
\begin{equation}\label{eq:lower_bound}
v(t) \geqslant e^{-\frac{n-2s}{2}}\, v(0) \quad {\rm for \ all} \quad t \in (-1,1).
\end{equation}
Next, multiplying \eqref{eq_log-cyl} by $\eta$ and integrating, we can invoke the lower bound in \eqref{eq:lower_bound}, the self-adjointness property in
\eqref{eq:self_adjointness}, and the upper bound in \eqref{eq:harnack} to conclude
\begin{eqnarray*}
c_{n,s}\, e^{-\frac{n+2s}{2}}\, v(0)^{\frac{n+2s}{n-2s}} & \leqslant &
c_{n,s} \int_{-1}^1 v^{\frac{n+2s}{n-2s}}\, \eta\, dt = \int_{-\infty}^\infty \eta\,
(P_s(v) + c_{n,s} v)\, dt \\
& = & \int_{-\infty}^\infty v\, (P_s(\eta)+ c_{n,s} \eta)\, dt \leqslant
\int_{-\infty}^\infty v \left ( |P_s(\eta)| + c_{n,s}\, \eta \right ) dt \\
& \leqslant & v(0) \int_{-\infty}^\infty e^{\frac{n}{2}|t|} \left ( |P_s(\eta)| + c_{n,s}\, \eta \right ) dt \leqslant C_1 v(0),
\end{eqnarray*}
where
\[
C_1 := \sup_{s\in[\frac{3}{4},1)} \int_{-\infty}^\infty e^{\frac{n}{2}|t|} \left ( |P_s(\eta)| + c_{n,s}\, \eta \right ) dt .
\]

By \eqref{eq:uniform_bound}, and since $\supp \eta \subset (-1,1)$ and $\int_{-1}^{1} \eta\, dt = 1$, we have
\[
C_1 \leqslant \sup_{s\in[\frac{3}{4},1)} \int_{-\infty}^\infty e^{\frac{n}{2}|t|}\, |P_s(\eta)|\, dt
+ e^{\frac{n}{2}} \sup_{s\in[\frac{3}{4},1)} c_{n,s} < \infty.
\]
Next, dividing by $c_{n,s}\, e^{-\frac{n+2s}{2}}\, v(0)$ and using
\(e^{\frac{n+2s}{2}} \leqslant e^{\frac{n+2}{2}}\) for all \(s\in[\tfrac{3}{4},1)\),
we can rearrange this last inequality to read
$$
v(0) \leqslant \left ( \frac{e^{\frac{n+2}{2}}\, C_1 }{c_{n,s}} \right )^{\frac{n-2s}{4s}} =: C_2(s).
$$
Finally, since
\[
c_0 := \inf_{s\in[\frac{3}{4},1)}c_{n,s}>0 \quad {\rm and} \quad
\frac{n-2}{4} < \frac{n-2s}{4s} \leqslant \frac{2n-3}{6} \quad {\rm for \ all} \quad s\in[\tfrac{3}{4},1),
\]
one has
\[
C_*:=\sup_{s\in[\frac{3}{4},1)}C_2(s) \leqslant \max \left \{ 1, \frac{C_1}{c_0}\,e^{\frac{n+2}{2}} \right \}^{\frac{2n-3}{6}}  < \infty ,
\]
which gives us
\[
v(0) \leqslant C_* \quad {\rm for \ all} \quad v\in\mathcal{A}_s \quad {\rm with} \quad s\in[\tfrac{3}{4},1).
\]
Since $t_0 \in \R$ was arbitrary, the proof is complete.
\end{proof}

The uniform bound of Proposition~\ref{prop:uniform_bound} is the only {\it a priori} estimate we need in the proof of the
classification, to which we now turn.

\section{The proof of classification}
We now prove Theorem~\ref{main_thm}. The proof has two parts. In \S~\ref{sec:glob_uniq}, we prove a Cauchy uniqueness
property: if $s$ is close to one, two solutions in the class $\mathcal{A}_s$ of Definition~\ref{def:class} which have
the same value and the same derivative at one point coincide. This is proved by contradiction, through a compactness
argument for the normalized difference of two solutions. Either the limit profile solves a homogeneous second order ODE
with zero Cauchy data, and is therefore trivial, or, after recentering, it solves a linear ODE on the whole line and decays
too fast at one end, which is excluded by the phase-plane analysis of the limit equation. In \S~\ref{sec:reflection}, we apply
this property to a solution and its reflection about a critical point, and deduce the classification. Before that, in
\S~\ref{sec:loc_uniq}, we collect the three tools on which the compactness argument rests.

\subsection{Preliminaries}\label{sec:loc_uniq}
In this subsection, we first introduce a weighted space suitable for studying our linearized equation. Then we prove
an interior estimate which is uniform as $s \nearrow 1$, the compactness of the class $\mathcal{A}_s$ as $s \nearrow 1$, and
the properties of the limit equation that we need.

\begin{definition}
Let $n \in \N$ with $n \geqslant 3$. For $\sigma = \frac{n}{2}$, let us define the weighted norm
\[
\| \phi \|_{*}:= \sup_{t\in\R} | e^{- \sigma |t|} \phi (t)|.
\]
\end{definition}

\begin{remark}
The exponent $\sigma$ is chosen so that
\[
\frac{n-2}{2} < \sigma < \frac{n+2s}{2} .
\]
The first inequality is used in Lemma~\ref{lem:limit_ode} below, and the second guarantees that $P_s(\phi)$ is well defined
whenever $\|\phi\|_*<\infty$.
\end{remark}

We begin with the interior estimate.

\begin{lemma}\label{lm:regularityC3}
Let $n \in \N$ with $n \geqslant 3$ and let $M \geqslant 1$. There exist $s_0 \in [\frac{3}{4},1)$, $\alpha \in (0,1)$ and $C>0$,
depending only on $n$ and $M$, such that the following holds for every $s \in [s_0,1)$. If
$\phi \in \mathcal{C}^2(\R)$ satisfies $\|\phi\|_* < \infty$ and
\[
P_s(\phi) = \varphi_s \quad {\rm in} \quad I_M := (-M-2,M+2),
\]
then
\begin{equation}\label{eq:C3a-estimate}
\|\phi\|_{\mathcal{C}^{1,\alpha}([-M,M])} \leqslant C \left ( \|\phi\|_{*} + \|\varphi_s\|_{L^\infty(I_M)} \right ).
\end{equation}
\end{lemma}

\begin{proof}
Let $\chi \in \mathcal{C}^\infty_c(\R)$ satisfy $0 \leqslant \chi \leqslant 1$, $\chi = 1$ on $[-M-2,M+2]$ and
$\supp \chi \subset [-M-3,M+3]$, and set $\widehat \phi = \chi \phi$, so that
\[
\|\widehat \phi\|_{L^\infty(\R)} \leqslant e^{\sigma(M+3)}\, \|\phi\|_* .
\]
We first compute the equation satisfied by $\widehat\phi$. For $|t| < M+1$, one has $\widehat\phi(t) = \phi(t)$, and so
\[
P_s(\widehat\phi)(t) = \varphi_s(t) + \int_{|t-\tau| \geqslant 1} (1-\chi(\tau))\, \phi(\tau)\, K_s(t-\tau)\, d\tau ,
\]
where the integral is restricted to $|t-\tau| \geqslant 1$ because $1-\chi(\tau)$ vanishes for $|\tau| \leqslant M+2$.
Since $|\phi(\tau)| \leqslant e^{\sigma(M+1)} e^{\sigma|t-\tau|} \|\phi\|_*$, the last integral is bounded by
\[
e^{\sigma(M+1)}\, \|\phi\|_* \int_{|\rho| \geqslant 1} e^{\sigma|\rho|}\, K_s(\rho)\, d\rho \leqslant C \|\phi\|_* ,
\]
by Lemma~\ref{lem:kernel}(iii). Next, we compare $P_s$ with the one-dimensional fractional Laplacian. We write
\[
K_s(\rho) = \varsigma_{n,s} |\rho|^{-(1+2s)} + \widetilde r_s(\rho) \quad {\rm for \ all} \quad \rho \neq 0 .
\]
By Lemma~\ref{lem:kernel}, one has
\[
\int_{|\rho| \leqslant 1} |\widetilde r_s(\rho)|\, d\rho \leqslant C(1-s) \int_{|\rho| \leqslant 1} \left ( 1 + |\rho|^{1-2s} \right ) d\rho
= C(1-s) \left ( 2 + \frac{1}{1-s} \right )
\]
and
\[
\int_{|\rho| \geqslant 1} |\widetilde r_s(\rho)|\, d\rho \leqslant \int_{|\rho| \geqslant 1} K_s(\rho)\, d\rho
+ \varsigma_{n,s} \int_{|\rho| \geqslant 1} |\rho|^{-(1+2s)}\, d\rho \leqslant C(1-s) .
\]
Therefore, we find $\|\widetilde r_s\|_{L^1(\R)} \leqslant C$ uniformly in $s \in [\frac{3}{4},1)$.

Since $\varsigma_{n,s} |\rho|^{-(1+2s)}$ is exactly the kernel of the one-dimensional fractional
Laplacian $(-\partial_t^2)^s$, we conclude that
\[
(-\partial_t^2)^s \widehat \phi = g \quad {\rm in} \quad (-M-1,M+1),
\]
where
\begin{equation} \label{eq:explicit_g}
g(t) := \varphi_s(t) + \int_{|t-\tau| \geqslant 1} (1-\chi(\tau))\, \phi(\tau)\, K_s(t-\tau)\, d\tau
- \int_{-\infty}^\infty \big( \widehat\phi(t) - \widehat\phi(\tau) \big)\, \widetilde r_s(t-\tau)\, d\tau .
\end{equation}
In addition, one has 
\[
\int_{-\infty}^\infty \big( \widehat\phi(t) - \widehat\phi(\tau) \big)\, \widetilde r_s(t-\tau)\, d\tau\leqslant 2 \|\widehat\phi\|_{L^\infty(\R)} \|\widetilde r_s\|_{L^1(\R)} \leqslant C \|\phi\|_*
\]
and so it converges absolutely.
Therefore
\[
\|g\|_{L^\infty(-M-1,M+1)} \leqslant C \left ( \|\phi\|_* + \|\varphi_s\|_{L^\infty(I_M)} \right ).
\]
Finally, the estimate \eqref{eq:C3a-estimate} follows from the interior $\mathcal{C}^{1,\alpha}$ regularity for the
one-dimensional fractional Laplacian with bounded right-hand side, applied to $\widehat\phi$, which coincides with $\phi$ on $[-M,M]$.
Precisely, \cite[Theorem~5.2]{MR2244602} yields the $\mathcal{C}^\alpha$ estimate, and
\cite[Theorem~11.1]{CS_regularity} yields the $\mathcal{C}^{1,\alpha}$ estimate, with constants
that remain bounded as $s \nearrow 1$ by the uniform-in-order design of that theory
(see also \cite{MR3542618}).
\end{proof}

Next, we record the compactness of the class $\mathcal{A}_s$ as $s \nearrow 1$. This is where the uniform bound of
Proposition~\ref{prop:uniform_bound} enters.

\begin{lemma} \label{lem:compactness}
Let $n \in \N$ with $n \geqslant 3$. Let $s_k \nearrow 1$ as $k \to \infty$, let $v_k \in \mathcal{A}_{s_k}$ and let $\tau_k \in \R$.
Then, after passing to a subsequence, $v_k(\cdot + \tau_k) \to v_\infty$ in $\mathcal{C}^1_{\rm loc}(\R)$, where
$v_\infty \in \mathcal{C}^2(\R)$ is a bounded, nonnegative solution to
\begin{equation} \label{limit1}
- v_\infty'' + c_{n,1}\, v_\infty = c_{n,1}\, v_\infty^{\frac{n+2}{n-2}} \quad {\rm in} \quad \R.
\end{equation}
\end{lemma}

\begin{proof}
Since $\mathcal{A}_s$ is invariant under translations, we may assume $\tau_k = 0$. By Proposition~\ref{prop:uniform_bound},
$0 < v_k \leqslant C_*$, and so
\[
P_{s_k}(v_k) = c_{n,s_k} \left ( v_k^{\frac{n+2s_k}{n-2s_k}} - v_k \right )
\]
is bounded uniformly in $k$. By Lemma~\ref{lm:regularityC3} and the Arzel\`a-Ascoli theorem, we extract a subsequence converging in
$\mathcal{C}^1_{\rm loc}(\R)$ to some $v_\infty$ with $0 \leqslant v_\infty \leqslant C_*$. Let $\eta \in \mathcal{C}^\infty_c(\R)$.
By \eqref{eq:uniform_bound} in Lemma 9, we get
\[
\int_{-\infty}^\infty v_k \left ( P_{s_k}(\eta) + c_{n,s_k}\, \eta \right ) dt = c_{n,s_k} \int_{-\infty}^\infty v_k^{\frac{n+2s_k}{n-2s_k}}\, \eta\, dt ,
\]
and \eqref{eq:convergence} in Lemma 9, we can pass to the limit to obtain
\[
\int_{-\infty}^\infty v_\infty \left ( -\eta'' + c_{n,1}\, \eta \right ) dt = c_{n,1} \int_{-\infty}^\infty v_\infty^{\frac{n+2}{n-2}}\, \eta\, dt .
\]
Thus $v_\infty$ solves \eqref{limit1} in the sense of distributions, and since $v_\infty$ is continuous, it is a classical solution.
\end{proof}

Finally, we collect the facts about the limit equation \eqref{limit1} that we need. 
Here and below we set
$\sigma_* := \frac{n-2}{2}$, so that $c_{n,1} = \sigma_*^2$, and we denote by
\[
v_{\circ}(t) = \hat{c}_{n,1}\, (\cosh t)^{\frac{2-n}{2}} \quad {\rm with} \quad \hat{c}_{n,1} = \left ( \frac{n}{n-2} \right )^{\frac{n-2}{4}}
\]
the spherical solution of order one (see \eqref{normalization_limit}).

The results collected in the following lemma are well-known. We recommend for instance the introduction 
of \cite{CGS}, \cite[Section 2]{MR994021} or \cite[Section 2.1]{MR1666838}. 
\begin{lemma} \label{lem:limit_ode}
Let $n \in \N$ with $n \geqslant 3$.
\begin{itemize}
\item[{\rm (i)}] Each bounded, nonnegative solution $v_\infty \in \mathcal{C}^2(\R)$ to \eqref{limit1} is one of the following:
$v_\infty \equiv 0$; $v_\infty \equiv 1$; a translate of the spherical solution $v_\circ$; or a
positive, nonconstant periodic solution. The energy
\[
\mathscr{H}(v_\infty) := \frac{1}{2} (v_\infty')^2 - c_{n,1} \left ( \frac{1}{2}\, v_\infty^2 - \frac{n-2}{2n}\, v_\infty^{\frac{2n}{n-2}} \right )
\]
is constant along each solution. It vanishes exactly for $v_\infty \equiv 0$ and the translates of $v_\circ$, and each level
$\mathscr{H} \in [-\frac{c_{n,1}}{n},0)$ contains exactly one of the remaining solutions, up to translation.
\item[{\rm (ii)}] Let $v_\infty, \overline{v}_\infty$ be bounded, nonnegative solutions to \eqref{limit1}. If there exist $C>0$ and
$\mu > \sigma_*$ such that
\[
|v_\infty(\zeta) - \overline{v}_\infty(\zeta)| \leqslant C e^{-\mu \zeta} \quad {\rm for \ all} \quad \zeta \geqslant 0,
\]
then $v_\infty \equiv \overline{v}_\infty$.
\item[{\rm (iii)}] Let $v_\infty$ be a bounded, nonnegative solution to \eqref{limit1} and let $\psi \in \mathcal{C}^2(\R)$ solve
\begin{equation} \label{eq:limit2}
-\psi'' + \frac{(n-2)^2}{4}\, \psi - \frac{n^2-4}{4}\, v_\infty^{\frac{4}{n-2}}\, \psi = 0 \quad {\rm in} \quad \R .
\end{equation}
If there exist $C>0$ and $\mu > \sigma_*$ such that
\[
|\psi(\zeta)| \leqslant C e^{-\mu \zeta} \quad {\rm for \ all} \quad \zeta \geqslant 0,
\]
then $\psi \equiv 0$.
\end{itemize}
\end{lemma}

\begin{proof}
To prove (i) we use the classical phase-plane analysis of \eqref{limit1} (see \cite{fowler} and \cite{MR1666838}). We recall that
the function
\[
W(z) := -c_{n,1} \left ( \frac{1}{2}\, z^2 - \frac{n-2}{2n}\, z^{\frac{2n}{n-2}} \right )
\]
vanishes at $z=0$, is strictly decreasing on $[0,1]$ with $W(1) = -\frac{c_{n,1}}{n}$, and is strictly increasing on $[1,\infty)$.
Hence, in the half-plane $\{ z \geqslant 0 \}$, each level set
\[
\left \{ (z,z') : \tfrac{1}{2}(z')^2 + W(z) = \mathscr{H} \right \} \quad {\rm with} \quad -\frac{c_{n,1}}{n} < \mathscr{H} < 0
\]
is a single closed orbit around $(1,0)$, the level $\mathscr{H} = 0$ consists of the origin and the homoclinic orbit of $v_\circ$,
and the orbits with $\mathscr{H}>0$ leave $\{ z \geqslant 0 \}$.

To prove (ii), let $\phi_\infty = v_\infty - \overline{v}_\infty$. By hypothesis, $\phi_\infty(\zeta) \to 0$ as $\zeta \to -\infty$, and by \eqref{limit1} the
function $\phi_\infty''$ is bounded and tends to zero as well, so $\phi_\infty'(\zeta) \to 0$. Therefore $\mathscr{H}(v_\infty) = \mathscr{H}(\overline{v}_\infty)$.
If this common value is negative, then by (i) either $v_\infty \equiv \overline{v}_\infty \equiv 1$, or
$\overline{v}_\infty = v_\infty(\cdot + t_0)$ with $v_\infty$ periodic, in which case $\phi_\infty$ is periodic and tends to zero at $-\infty$,
hence $\phi_\infty \equiv 0$. If the common value is zero, then $v_\infty$ and $\overline{v}_\infty$ belong to
$\{0\} \cup \{ v_\circ(\cdot - t_0) : t_0 \in \R\}$. Since
\[
v_\circ(\zeta - t_0) = \hat{c}_{n,1}\, 2^{\sigma_*} e^{-\sigma_* t_0}\, e^{\sigma_* \zeta} \left ( 1 + \mathcal{O}(e^{2\zeta}) \right )
\quad {\rm as} \quad \zeta \to -\infty,
\]
the absolute value of the difference of two distinct elements of this set is bounded from below by a positive multiple of
$e^{\sigma_* \zeta}$ near $-\infty$, which is incompatible with the hypothesis because $\mu > \sigma_*$. Hence
$v_\infty \equiv \overline{v}_\infty$.

To prove (iii), we consider the four cases in (i).

\noindent{\it Case 1:} $v_\infty \equiv 0$.

\noindent In this case, one has
\[
\psi(\zeta) = A e^{\sigma_* \zeta} + B e^{-\sigma_* \zeta} \quad {\rm for \ some} \quad A,B \in \R,
\]
and the bound forces $B = 0$ and then $A = 0$.

\noindent{\it Case 2:} $v_\infty \equiv 1$.

\noindent In this case, \eqref{eq:limit2} reads $\psi'' + (n-2) \psi = 0$, whose nontrivial solutions do not tend to zero.

\noindent{\it Case 3:} $v_\infty = v_\circ(\cdot - t_0)$ for some $t_0 \in \R$.

\noindent In this case, $\psi_+ := v_\infty'$ solves \eqref{eq:limit2}, is positive on $(-\infty,t_0)$ and satisfies
\[
\psi_+(\zeta) = A_+ e^{\sigma_*\zeta}(1+o(1)) \quad {\rm as} \quad \zeta \to -\infty \quad {\rm with} \quad A_+>0.
\]
A second solution is
\[
\psi_-(\zeta) := \psi_+(\zeta) \int_\zeta^{a-1} \psi_+(\xi)^{-2}\, d\xi\]
which \ satisfies
\[
\psi_-(\zeta) = A_- e^{-\sigma_* \zeta}(1+o(1)) \quad {\rm with} \quad A_->0.
\]
Writing $\psi = A \psi_+ + B \psi_-$, the bound again forces $B=0$ and then $A=0$.

\noindent{\it Case 4:} $v_\infty$ is periodic and nonconstant, with period $T>0$.

\noindent In this case, one has $v_\infty = v_\tau$ is the Delaunay 
solution when $s=1$ (the scalar curvature case)  with necksize $\tau>0$ and $T = T_\tau$ is its minimal period. We can 
differentiate the relation 
$$v_\tau(t+T_\tau) = v_\tau(t)$$ 
with respect to $t$ to obtain one solution $\psi_+ = \partial_tv_\tau$ of the linearized 
equation, which is bounded and periodic. We differentiate with respect to 
the necksize parameter to obtain a second solution $\psi_- = \partial_\tau v_\tau$, which grows linearly. Once again, the only linear combination of 
these two solutions satisfying the bound is the zero function.
\end{proof}

With these three tools, we are ready to prove the Cauchy uniqueness property.

\subsection{Cauchy uniqueness} \label{sec:glob_uniq}
In this subsection, we prove the main technical result of the paper. In the scalar curvature setting, it is the
uniqueness theorem for the initial value problem of a second order ODE. In the nonlocal setting there is no initial value problem,
and we recover the statement only for $s$ close to one.

\begin{proposition} \label{prop:cauchy_uniq}
Let $n \in \N$ with $n \geqslant 3$. There exists $\delta \in (0,\frac{1}{4})$ such that the following holds for every
$s \in (1-\delta,1)$. If $v, \overline{v} \in \mathcal{A}_s$ satisfy
\[
v(t_0) = \overline{v}(t_0) \quad {\rm and} \quad v'(t_0) = \overline{v}'(t_0) \quad {\rm for \ some} \quad t_0 \in \R ,
\]
then $v \equiv \overline{v}$.
\end{proposition}

\begin{proof}
We argue by contradiction. If the statement is false, there exist $s_k \nearrow 1$ and $v_k, \overline{v}_k \in \mathcal{A}_{s_k}$
with the same Cauchy data at some point, which after a translation we take to be the origin, such that
$\phi_k := v_k - \overline{v}_k \not\equiv 0$. By Proposition~\ref{prop:uniform_bound}, it holds
\[
0 < v_k \leqslant C_* \quad {\rm and} \quad 0 < \overline{v}_k \leqslant C_* \quad {\rm for \ all} \quad k \in \N .
\]
We divide the proof into two steps and two cases.

\noindent{\it Step 1: The linearized equation.}

\noindent Since $v_k$ and $\overline{v}_k$ satisfy \eqref{eq_log-cyl} with $s = s_k$, the difference 
$\phi_k = \overline{v}_k - v_k$ satisfies
\begin{equation}\label{eq:linearizedequationk}
   \begin{cases}
       P_{s_k}(\phi_k) + c_{n,s_k} \phi_k - V_{s_k}(t)\, \phi_k= 0 \quad {\rm in} \quad \R\\
       \phi_k (0)=\phi^{'}_k (0)=0,
   \end{cases}
\end{equation}
where
\[
V_{s_k}(t) := c_{n,s_k}\, G_{k} \big( \overline{v}_k(t), v_k(t) \big) \quad {\rm and} \quad
G_{k}(a,b) := \frac{n+2s_k}{n-2s_k} \int_0^1 \big( a + \theta (b-a) \big)^{\frac{4s_k}{n-2s_k}}\, d\theta .
\]
The equation is exact, with no remainder, because
\[
b^{\frac{n+2s_k}{n-2s_k}} - a^{\frac{n+2s_k}{n-2s_k}} = G_{k}(a,b)\,(b-a) \quad {\rm for \ all} \quad a,b \geqslant 0.
\]
The functions $G_k:[0,\infty)\times[0,\infty)\to \mathbf{R}$ and converge locally uniformly, as $k \to \infty$, to
\[
G_\infty(a,b) := \frac{n+2}{n-2} \int_0^1 \big( a + \theta (b-a) \big)^{\frac{4}{n-2}}\, d\theta ,
\]
which satisfies
\[
G_\infty(a,a) = \frac{n+2}{n-2}\, a^{\frac{4}{n-2}}.
\]
In particular, $0 \leqslant V_{s_k} \leqslant C$ with $C$ independent of $k$.

\noindent{\it Step 2: Normalization.}

\noindent Since $\phi_k$ is bounded and not identically zero, $\|\phi_k\|_* \in (0,\infty)$ and the supremum is attained at some
$t_k \in \R$. Exchanging the roles of $v_k$ and $\overline{v}_k$, and replacing both by their reflections $t \mapsto -t$ if
necessary, we may assume $\phi_k(t_k)>0$ and $t_k \geqslant 0$. These operations preserve the class $\mathcal{A}_{s_k}$ and the
Cauchy data at the origin. By linearity, we may, without loss of generality, assume that $\|\phi_k\|_*=1$, since the normalization $\hat{\phi}_k=\tfrac{\phi_k}{\|\phi_k\|_*}$,
which still solves \eqref{eq:linearizedequationk}. Thus, it holds
\begin{equation}\label{eq:firstbound}
    \phi_k (t_k)= e^{\sigma t_k}\quad {\rm and} \quad |\phi_k (t)| \leqslant e^{\sigma |t|}
    \quad {\rm for \ all} \quad t\in\R .
\end{equation}

After passing to a subsequence, we consider two cases, as follows:

\noindent{\it Case 1:} $\limsup_{k\to\infty}|t_k|<\infty$.

\noindent In this case, up to a subsequence, $t_k\to t_\infty$ for some finite $t_\infty \geqslant 0$. By Lemma~\ref{lem:compactness},
after passing to a further subsequence, $v_k \to v_*$ and $\overline{v}_k \to \overline{v}_*$
in $\mathcal{C}^1_{\rm loc}(\R)$, where $v_*$ and $\overline{v}_*$ are bounded, nonnegative solutions to \eqref{limit1}. They
have the same Cauchy data at the origin, and since $z \mapsto z^{\frac{n+2}{n-2}}$ is locally Lipschitz on $[0,\infty)$, uniqueness for the
initial value problem implies $v_* \equiv \overline{v}_*$. Consequently, we get
\[
\left\|V_{s_k} - \frac{n^2-4}{4}\, (\overline{v}_*)^{\frac{4}{n-2}}\right\|_{\mathcal{C}^1(\mathbb{R})}\to 0 \quad {\rm as} \quad k\to\infty
\]
By \eqref{eq:linearizedequationk} and \eqref{eq:firstbound}, the family $\{P_{s_k}(\phi_k)\}_{k\in\mathbf{N}}$ with
\(
P_{s_k}(\phi_k) = \left ( V_{s_k} - c_{n,s_k} \right ) \phi_k
\)
is locally uniformly bounded. 
By Lemma~\ref{lm:regularityC3}, we can find a subsequence such that
$\phi_k \to \phi_\infty$ in $\mathcal{C}^1_{\rm loc}(\R)$, and $|\phi_\infty(t)| \leqslant e^{\sigma|t|}$. For
$\eta \in \mathcal{C}^\infty_c(\R)$, by Lemma~\ref{lem:test}(c), it follows
\[
\int_{-\infty}^\infty \phi_k \left ( P_{s_k}(\eta) + c_{n,s_k} \eta - V_{s_k}\, \eta \right ) dt = 0 ,
\]
and by \eqref{eq:firstbound} and Lemma~\ref{lem:test}(b), we can pass to the limit. Thus $\phi_\infty \in \mathcal{C}^2(\R)$ and
\begin{equation}\label{eq:limitlinearizedequationk}
    \begin{cases}
        -\phi_\infty^{''}+\frac{(n-2)^2}{4} \phi_\infty - \frac{(n^2-4)}{4} (\overline{v}_*(t)
        )^{\frac{4}{n-2}}\phi_\infty=0 \quad {\rm in} \quad \R\\
        \phi_\infty(0)=\phi_\infty^{'}(0)=0,
    \end{cases}
\end{equation}
where the Cauchy data pass to the limit by the $\mathcal{C}^1_{\rm loc}$ convergence. By ODE uniqueness applied to
\eqref{eq:limitlinearizedequationk}, one has $\phi_\infty\equiv 0$. On the other
hand, \eqref{eq:firstbound} and local uniform convergence yield
\[
\phi_\infty(t_\infty)=\lim_{k\to\infty}\phi_k(t_k) = e^{\sigma t_\infty} \geqslant 1,
\]
a contradiction.

\noindent{\it Case 2:} $\limsup_{k\to\infty}|t_k|=\infty$.

\noindent In this case, up to a subsequence, $t_k \to \infty$. We consider the rescaled function
\[
\psi_k (\zeta):= \frac{ \phi_k(t_k+\zeta)}{\phi_k (t_k)},
\]
which satisfies $\psi_k (0)=1$ and, by \eqref{eq:firstbound}, we get
\begin{equation} \label{eq:psi_k_bound}
|\psi_k(\zeta)| \leqslant e^{\sigma (|t_k + \zeta| - t_k)} \leqslant e^{\sigma|\zeta|} \quad {\rm for \ all} \quad \zeta \in \R
\quad {\rm and} \quad |\psi_k(\zeta)| \leqslant e^{\sigma \zeta} \quad {\rm for} \quad \zeta\in(-t_k, \infty) .
\end{equation}
In addition, one has
\begin{equation}\label{eq:psi_k_eq}
P_{s_k}(\psi_k)(\zeta) + c_{n,s_k}\psi_k(\zeta) - V_{s_k}(t_k+\zeta)\,\psi_k(\zeta) = 0 \quad {\rm in} \quad \R .
\end{equation}
As in Case 1, by Lemma~\ref{lm:regularityC3} we extract a subsequence with $\psi_k \to \psi_\infty$ in $\mathcal{C}^1_{\rm loc}(\R)$.
Since $t_k \to \infty$, for any fixed $\zeta \in \R$, $\zeta > -t_k$ for all $k\gg1$, and so
\eqref{eq:psi_k_bound} yields
\begin{equation}\label{eq:exponentialboundpsi}
\psi_\infty(0) = 1 \quad {\rm and} \quad |\psi_\infty(\zeta)| \leqslant e^{\sigma\zeta} \quad {\rm for \ all} \quad \zeta \in \R .
\end{equation}
By Lemma~\ref{lem:compactness}, after passing to a further subsequence,
\[
v_k(t_k + \cdot) \to w_* \quad {\rm and} \quad \overline{v}_k(t_k + \cdot) \to \overline{w}_* \quad {\rm in} \quad \mathcal{C}^1_{\rm loc}(\R),
\]
where $w_*$ and $\overline{w}_*$ are bounded, nonnegative solutions to \eqref{limit1}. 

We first show that these two limits coincide,
and then we conclude.

\noindent{\it Claim 1:} $w_* \equiv \overline{w}_*$.

\noindent For any $k\in\mathbf{N}$, the function $\psi_k$ is a multiple of $\phi_k$, and so it solves the 
same linear equation. 
Let $\psi_* := w_* - \overline{w}_*$, so that the numerators converge to $\psi_*$ locally uniformly and the
denominators converge to $\psi_*(0)$, and suppose that $\psi_* \not\equiv 0$. If $\psi_*(0) = 0$, we choose $\zeta_1 \in \R$ with $\psi_*(\zeta_1) \neq 0$
and obtain $|\psi_k(\zeta_1)| \to \infty$, which contradicts \eqref{eq:psi_k_bound}. Hence $\psi_*(0) \neq 0$,
and then $\psi_\infty = \tfrac{\psi_*}{\psi_*(0)}$. By \eqref{eq:exponentialboundpsi}, it folowws
\[
|w_*(\zeta) - \overline{w}_*(\zeta)| \leqslant |\psi_*(0)|\, e^{\sigma \zeta} \quad {\rm for \ all} \quad \zeta \in \R .
\]
Since $\sigma > \sigma_*$, Lemma~\ref{lem:limit_ode}(b) implies $w_* \equiv \overline{w}_*$, a contradiction.

\noindent{\it Claim 2:} $\psi_\infty \equiv 0$.

\noindent By Claim 1 and the convergence of $G_k$, one has
\[
V_{s_k}(t_k + \zeta) \longrightarrow V_\infty(\zeta) := \frac{n^2-4}{4}\, (w_*(\zeta))^{\frac{4}{n-2}} \quad {\rm locally \ uniformly \ in} \quad \R.
\]
Testing \eqref{eq:psi_k_eq} against $\eta \in \mathcal{C}^\infty_c(\R)$
and passing to the limit as in Case 1, now using the first bound in \eqref{eq:psi_k_bound}, we find that
$\psi_\infty \in \mathcal{C}^2(\R)$ solves \eqref{eq:limit2} with $v_\infty = w_*$. By \eqref{eq:exponentialboundpsi} and
Lemma~\ref{lem:limit_ode}(c), we get $\psi_\infty \equiv 0$.

Claim 2 contradicts $\psi_\infty(0) = 1$, which finishes the proof.
\end{proof}

\begin{remark}
In the proof above, the comparison solution $\overline{v}_k$ plays no special role. In particular, we do not need to know
that it is a Delaunay solution, nor that it has the same supremum as $v_k$. Claim 1 is needed because the potential
$V_{s_k}$ is a difference quotient of two solutions, which {\it a priori} may have different limits along the sequence $\{t_k\}_{k \in \N}$.
\end{remark}

We now turn to the consequences of Proposition~\ref{prop:cauchy_uniq}.

\subsection{Global classification} \label{sec:reflection}
In this subsection, we complete the proof of Theorem~\ref{main_thm} by 
applying Proposition~\ref{prop:cauchy_uniq} to a
solution and its reflection about a critical point. This is the classical argument for 
second order equations, in which
Proposition~\ref{prop:cauchy_uniq} plays the role of the uniqueness theorem for 
the initial value problem. The same scheme was
used by Frank and K\"onig \cite[Lemma~8 and Proposition~3]{FrankKonig} for a fourth 
order equation, where the uniqueness
statement is due to van den Berg \cite{vandenBerg}.
We need two more elementary facts, valid for each fixed $s \in [\frac{3}{4},1)$.

The first point in the lemma below appears as  \cite[Theorem 1.1]{Azahara_Sergio_Ruiz}; we 
include a proof below for the reader's convenience. 
\begin{lemma} \label{lem:liouville}
Let $n \in \N$ with $n \geqslant 3$ and $s \in [\frac{3}{4},1)$, and 
let $v \in \mathcal{C}^2(\R)$ be a bounded, positive solution
to \eqref{eq_log-cyl}.
\begin{itemize}
\item[{\rm (i)}] If $v \geqslant 1$ on $\R$, or $v \leqslant 1$ on $\R$, then $v \equiv 1$.
\item[{\rm (ii)}] If $v$ is monotone on a half-line, then its limit at the corresponding end is
either $0$ or $1$.
\end{itemize}
\end{lemma}

\begin{proof}
For (i), let $\eta \in \mathcal{C}^\infty_c(\R)$ with $0 \leqslant \eta \leqslant 1$, $\eta = 1$ on $[-1,1]$ and
$\supp \eta \subset [-2,2]$, and let $\eta_R(t) = \eta(t/R)$ for $R \geqslant 1$. Since $\int_{-\infty}^\infty P_s(\eta_R)\, dt = 0$
by \eqref{eq:limit}, we obtain
\[
c_{n,s} \int_{-\infty}^\infty \eta_R \left ( v^{\frac{n+2s}{n-2s}} - v \right ) dt = \int_{-\infty}^\infty \eta_R\, P_s(v)\, dt
= \int_{-\infty}^\infty (v-1)\, P_s(\eta_R)\, dt .
\]
Using \eqref{cyl_GJMS_sym} we have the estimate 
\[
|P_s(\eta_R)(t)| \leqslant \frac{\|\eta''\|_{L^\infty(\R)}}{2R^{2}} \int_{-\infty}^\infty \rho^2 K_s(\rho)\, d\rho \quad {\rm for \ all} \quad t \in \R.
\]
Combining this with Lemma~\ref{lem:kernel}(ii), we obtain 
\[
|P_s(\eta_R)(t)| \leqslant C R\, e^{-\frac{n+2s}{2}(|t| - 2R)} \quad {\rm for} \quad |t| \geqslant 3R ,
\]
and so it holds 
\[
\int_{-\infty}^\infty |P_s(\eta_R)|\, dt \leqslant \frac{C}{R}. 
\]
The right-hand side of the identity above tends to zero as $R \to \infty$. 
If $v \geqslant 1$, the
integrand on the left is nonnegative, and so
\[
0 \leqslant c_{n,s} \int_{-1}^1 \left ( v^{\frac{n+2s}{n-2s}} - v \right ) dt \leqslant c_{n,s} \int_{-\infty}^\infty \eta_R \left ( v^{\frac{n+2s}{n-2s}} - v \right ) dt
\longrightarrow 0 \quad {\rm as} \quad R \to \infty.
\]
Therefore $v^{\frac{n+2s}{n-2s}} = v$ on $[-1,1]$, and by translation invariance on all of $\R$, that is,
$v \equiv 1$. The case $v \leqslant 1$ is the same.

For (ii), suppose $v$ is monotone on a half-line of the form $[T_0,\infty)$ or $(-\infty,T_0]$ for some $T_0>0$.
Without loss of generality, we may assume that
$v$ is monotone on $[T_0,\infty)$, which implies the existence of the limit $\ell = \lim_{t \to \infty} v(t) \in [0,\infty)$. For
$\eta \in \mathcal{C}^\infty_c(\R)$ and $T \in \R$, from \eqref{eq:limit} we obtain
\[
\int_{-\infty}^\infty v(t+T) \left ( P_s(\eta) + c_{n,s} \eta \right ) dt = c_{n,s} \int_{-\infty}^\infty v(t+T)^{\frac{n+2s}{n-2s}}\, \eta\, dt .
\]
Letting $T \to \infty$ and using dominated convergence and $\int_{-\infty}^\infty P_s(\eta)\, dt = 0$, we get
$\ell = \ell^{\frac{n+2s}{n-2s}}$, and so $\ell \in \{0,1\}$.
\end{proof}

We are now ready to prove the global uniqueness.

\begin{proof}[Proof of Theorem~\ref{main_thm}]
Let $\delta>0$ be given by Proposition~\ref{prop:cauchy_uniq}, let $s \in (1-\delta,1)$, let $u$ be a singular solution to
\eqref{const_q_euc} and let $v = \mathfrak{F}(u)$. By rotational symmetry $v$ depends only on $t$, which together with  Lemma~\ref{lem:gradient} implies that $v\in \mathcal{A}_s$.
Also, using Proposition~\ref{prop:uniform_bound}, it satisfies
\[
0 < \inf_{t \in \R} v(t) \leqslant \sup_{t \in \R} v(t) \leqslant C_* .
\]
Let $Z := \{ t \in \R : v'(t) = 0 \}$, which is a closed set. We divide the proof into four steps.

\noindent{\it Step 1: The function $v$ is even about each of its critical points.}

\noindent Let $t_0 \in Z$ and $\overline{v}(t) := v(2t_0 - t)$. Then $\overline{v} \in \mathcal{A}_s$ and
\[
\overline{v}(t_0) = v(t_0) \quad {\rm and} \quad \overline{v}'(t_0) = -v'(t_0) = 0 = v'(t_0).
\]
By Proposition~\ref{prop:cauchy_uniq}, we conclude $\overline{v} \equiv v$.

\noindent{\it Step 2: If $v \not\equiv 1$, then $Z$ contains at least two points.}

\noindent If $Z = \emptyset$, then $v$ is strictly monotone
on $\R$. By Lemma~\ref{lem:liouville}(ii) and $\inf_{t \in \R} v(t)>0$, both limits of $v$ at $\pm\infty$ equal one, which is impossible. If
$Z = \{t_0\}$, then $v$ is strictly monotone on $[t_0,\infty)$ and even about $t_0$ by Step 1. As before,
$v(t) \to 1$ as $t \to \pm\infty$, and so either $v > 1$ on $\R \backslash \{t_0\}$ or $v<1$ on $\R \backslash \{t_0\}$.
Lemma~\ref{lem:liouville}(i) yields $v \equiv 1$, a contradiction.

\noindent{\it Step 3: If $v \not\equiv 1$, then $v$ is periodic.}

\noindent If $Z$ contains an interval $(a,b)$, then by Step 1 the function $v$
is even about each point of $(a,b)$. Composing two such reflections shows that $v$ is periodic with every period in $(0,2(b-a))$,
hence constant, hence $v \equiv 1$. Otherwise, by Step 2 there are points $a<b$ in $Z$ and a point of $(a,b)$ not in $Z$. Let $t^* \in (a,b)$ with $t^* \not \in Z$ and let $(a',b')$ be the connected
component of $\R \backslash Z$ containing $t^*$. Then $a \leqslant a' < b' \leqslant b$, the points $a'$ and $b'$ belong to $Z$, and $v$ is
strictly monotone on $[a',b']$. By Step 1, $v$ is even about $a'$ and about $b'$, and therefore
\[
v(t + L) = v(t) \quad {\rm for \ all} \quad t \in \R, \quad {\rm with} \quad L := 2(b'-a') .
\]
In particular, $Z = a' + (b'-a') \Z$, and $v$ attains its minimum and its maximum exactly once in each period.
Lemma~\ref{lem:liouville}(i) then shows that
\[
\varepsilon := \min_{t \in \R} v(t) < 1 < \max_{t \in \R} v(t).
\]

\noindent{\it Step 4: The function $v$ is determined by its necksize.}

\noindent Let $v$ and $\overline{v}$ be the functions associated with two singular solutions, with
\[
\min_{t \in \R} v(t) = \min_{t \in \R} \overline{v}(t) = \varepsilon \in (0,1).
\]
Translating, we may assume that both minima are attained at the origin. Then
$v(0) = \overline{v}(0)$ and $v'(0) = \overline{v}'(0) = 0$, and Proposition~\ref{prop:cauchy_uniq} implies $v \equiv \overline{v}$.

In conclusion, $v \equiv 1$ or $v$ is, up to a translation, the unique even, positive, periodic solution to \eqref{eq_log-cyl} with
necksize $\varepsilon = \min_{t \in \R} v(t) \in (0,1)$, that is, the Delaunay solution $v_\varepsilon$. A translation in $t$ corresponds to a
dilation of $u$, which accounts for the parameter $R>0$ in the statement.
\end{proof}

We close this section with a remark on which necksizes occur.

\begin{remark} \label{rem:necksizes}
Let $\mathcal{E}_s \subset (0,1)$ be the set of necksizes of the nonconstant periodic solutions to \eqref{eq_log-cyl}. By
Lemma~\ref{lem:gradient}, every positive periodic solution belongs to $\mathcal{A}_s$, so the proof above applies to the
nonconstant minimizers of \eqref{del_min} constructed in \cite{MR3694655}. For $s \in (1-\delta,1)$, each of them is even
about its critical points and is the unique solution with its necksize. Theorem~\ref{main_thm} identifies the moduli space
$\mathcal{M}_{\{p,q\}}$ with $\mathcal{E}_s \cup \{1\}$. A compactness argument shows that $\mathcal{E}_s \cup \{1\}$ is closed in
$(0,1]$. We do not prove here that $\mathcal{E}_s = (0,1)$.
\end{remark}

\section{The proof of nondegeneracy} \label{sec:nondegen}

In this section we prove Theorem \ref{nondegen_thm}. Since the proof is very similar 
to that of Proposition~\ref{prop:cauchy_uniq}, we only indicate which changes are necessary. 
We first recall the two geometric Jacobi fields of a Delaunay solution, then we use them to reduce to zero Cauchy data, and finally we
run the compactness argument of \S~\ref{sec:glob_uniq}.

We let $s_k \nearrow 1$, let $\varepsilon_k \in (0,1]$ and let $v_k$ satisfy 
\eqref{eq_log-cyl} on $\R \times \Ss^{n-1}$. By Theorem~\ref{main_thm}, we know 
that $v_k$ is a Delaunay solution of order $s_k$, {\it i.e.} $v_k = v_{\varepsilon_k}$ 
for some $\varepsilon_k \in (0,1]$. We also let 
\[
\mathbb{L}_{k} := P_{s_k} + c_{n,s_k} 
- \tfrac{n+2s_k}{n-2s_k}\,c_{n,s_k}\,\bar v_{\varepsilon_k}^{\frac{4s_k}{n-2s_k}}
\]
be the linearization of \eqref{cyl_GJMS1} about $v_{\varepsilon_k}$ and let 
$w_k \in L^2(\R \times \Ss^{n-1})$ satisfy $\mathbb{L}_k(w_k) = 0$ 
with $\|w_k\|_{L^2(\R \times \Ss^{n-1})} = 1$.
Here we use two specific 
functions $\widetilde{w}_k^+$ and $\widetilde{w}_k^-$ in the kernel of the
operator $\mathbb{L}_k$, namely
$$\widetilde{w}_k^+ = \dot v_{\varepsilon_k} \quad {\rm and} \quad \widetilde{w}_k^-
= \left. \frac{d}{d\varepsilon} \right |_{\varepsilon_k} v_{\varepsilon} , $$
where the dot denotes the derivative with respect to $t$. 
If we let $T_\varepsilon$ be the period of the Delaunay solution with necksize 
$\varepsilon$ and differentiate the relation 
\begin{equation} \label{period_relation} 
v_\varepsilon (t+T_\varepsilon) = v_\varepsilon(t)\end{equation} 
with respect to $t$, we see that $\widetilde{w}_k^+$ is bounded and periodic. 
Differentiating \eqref{period_relation} with respect to $\varepsilon$ we 
see that $\widetilde{w}_k^-$ grows at most linearly. Furthermore, we translate each $v_\varepsilon$ so that its
minimum is attained at the origin, that is, $v_\varepsilon(0) = \varepsilon$ and $\dot v_\varepsilon(0) = 0$ for all
$\varepsilon$. Differentiating these two relations with respect to $\varepsilon$ yields
$$
\widetilde{w}_k^+ (0) = 0, \quad \dot{\widetilde{w}}_k^+(0) = \ddot v_{\varepsilon_k}(0), \quad
\widetilde{w}_k^-(0) = 1, \quad {\rm and} \quad \dot{\widetilde{w}}_k^-(0) = 0.
$$%

Now we use these two tempered, geometric Jacobi fields to adjust $w_k$.
Suppose first that $\ddot v_{\varepsilon_k}(0) \neq 0$. Then
we can choose constants $\alpha_k, \beta_k \in \R$ such that the function
$$\widehat{w}_k := w_k + \alpha_k \widetilde{w}_k^+ + \beta_k \widetilde{w}_k^-$$
satisfies $\widehat{w}_k(0) = 0$ and $\dot{\widehat{w}}_k(0) = 0$. Since $\mathbb{L}_k$ is linear and $\widetilde{w}_k^\pm \in \ker(\mathbb{L}_k)$, one has 
\begin{equation}
    P_{s_k} (\widehat{w}_k) + c_{n,s_k} \widehat{w}_k 
- \tfrac{n+2s_k}{n-2s_k}\,c_{n,s_k}\,\bar v_{\varepsilon_k}^{\frac{4s_k}{n-2s_k}} \widehat{w}_k=0, \widehat{w}_k (0)=0, \widehat{w}_k^{'} (0)=0.
\end{equation}
Since $\widetilde{w}_k^\pm$ grow at most linearly, $\|\widehat{w}_k\|_* < \infty$. We now argue as in the proof of
Proposition~\ref{prop:cauchy_uniq} replacing  $\phi_k$  by $\widehat{w}_k$ and considering the potential
\[
\widehat{V}_{s_k}(t):=\tfrac{n+2s_k}{n-2s_k}\,c_{n,s_k}\,\bar v_{\varepsilon_k}(t)^{\frac{4s_k}{n-2s_k}}
\]
Since only one solution
$\bar v_{\varepsilon_k}$ is involved, Claim 1 of Case 2 is not needed, and Lemma~\ref{lem:limit_ode}(c) yields the contradiction.
We deduce that $\widehat{w}_k\equiv 0$ for $k$ large, and hence
\[
{w}_k=- \alpha_k \widetilde{w}_k^+ - \beta_k \widetilde{w}_k^- .
\]
A nontrivial linear combination of $\widetilde{w}_k^+$ and $\widetilde{w}_k^-$ does not decay at infinity and so is not in $L^2(\mathbf{R})$.
Therefore $w_k \equiv 0$, which contradicts $\|w_k\|_{L^2(\mathbf{R})} = 1$. It remains to consider the case
$\ddot v_{\varepsilon_k}(0) = 0$ along a subsequence. In this case, $\dot v_{\varepsilon_k}$ is bounded, solves the same
equation and has zero Cauchy data at the origin, so the same argument shows that $\dot v_{\varepsilon_k} \equiv 0$ for $k$ large,
which is impossible when $\varepsilon_k < 1$. When $\varepsilon_k=1$ then $v_k\geqslant 1$, 
and hence invoking \cite[Theorem 1.1]{Azahara_Sergio_Ruiz}, we find $v_k \equiv 1$. We already know 
by explicit computations that $v\equiv 1$ is nondegenerate. 
\hfill $\square$

This completes the proof of our two main results. In the last section, we discuss extensions and applications.

\section{Extensions and applications} \label{sec:outlook} 
We conclude this paper by first discussing the possibility of extending our
classification to the case that $s>1$ and then discuss some possible applications. 

\subsection{Extending the range of the fractional parameter}\label{sec:extension}
It is quite natural to ask whether our classification result remains valid when the order of 
the conformal operator exceeds one. In this subsection we discuss our work in 
progress \cite{our_followup} to classify conformally flat, constant $Q_s$-curvature 
metrics on a twice-punctured sphere in the case that $s \in (1,1+\delta)$, where 
$\delta$ is sufficiently small. 

The first complication in this classification lies in the description of the 
Delaunay solutions of order $s$ when $s>1$. Indeed, the integral operator $P_s$
as given in \eqref{cyl_GJMS1} is not well-defined in the case $s>1$,  so one needs an 
alternative formulation of the problem. 
Jin and Xiong \cite{MR4266239} give a dual formulation of the problem, starting with the equation 
\begin{equation} \label{dual_form1}
u = c_{n,s} (-\Delta)^{-s} \left ( u^{\frac{n+2s}{n-2s}} \right ). 
\end{equation} 
After the Emden-Fowler change of variables they arrive at an integral 
operator with better convergence properties. 

The larger complication is that the passage from local convergence to 
global convergence. In particular, we would need to prove analogs of 
Lemmas \ref{lem:kernel} and \ref{lem:test} for the kernel of 
the dual operator. Deriving 
such estimates associated to the linearization of the dual operator 
will occupy a large portion of \cite{our_followup}.

\subsection{Informed speculation} \label{sec:applications}
As a direct consequence of our classification, one can rapidly prove  
several related results. First of all,  one can combine the nondegenerate with 
the slide-back technique in the paper by Korevaar, Mazzeo, Pacard and Schoen \cite{MR1666838}
to obtain refined asymptotics of solutions with isolated singularities. Second,
we expect that one can adapt the proof of compactness in the moduli space 
setting, as in the paper by Pollack \cite{MR1266101}. We only sketch the ideas of both 
of these proofs and leave the details for future projects. 

\subsubsection{Local behavior near isolated singularities} 
Our first application concerns the asymptotic behavior of a solution near an isolated singularity.

Using a blow-up analysis, Caffarelli, Jin, Sire and Xiong~\cite{MR3366748} show that any isolated singularity of a singular Yamabe metric on the sphere must 
be asymptotically radial. 
Combining this with our classification of the 
radial solutions, we see that any isolated singularity is asymptotic to a 
Delaunay solution. In other words, if $s< 1$ is sufficiently close to $1$ and 
if $\Lambda \subset \Ss^n$ is a finite set and $U: \Ss^n \backslash \Lambda 
\rightarrow (0,\infty)$ satisfies the boundary condition and solves the PDE 
\begin{equation} \label{const_q_sph_multiple}\tag{${\mathcal{Q}}^\circ_{n,s,\Lambda}$}
\mathbb{P}_s U = c_{n,s} U^{\frac{n+2s}{n-2s}} \quad {\rm on} \quad \mathbf{S}^n\setminus \Lambda. 
\end{equation} 
then for each $p_* \in \Lambda$ and sufficiently small radius $0<r\ll1$ there 
exists a (rescaled) Delaunay solution $\bar{U}_\varepsilon$ such that 
$$U(p) = \bar{U}_\varepsilon (p) (1+ o(1)) \quad \textrm{for all} \quad
x \in \mathbf{B}_r(p_*) \backslash \{ p_* \} \quad {\rm as} \quad p\to p_*. $$
Changing to the cylindrical gauge in the punctured ball $\mathbf{B}_r(p) 
\backslash \{ p \}$, we can write this estimate as 
\begin{equation} \label{simple_asymp}
|v(t,\theta) - v_\varepsilon(t+T)| = \mathrm{o}(1) \quad {\rm as} \quad t\to\infty  \end{equation}
for some $T \in \R$. 

In the case that $1-\delta < s < 1$ with $\delta>0$ sufficiently small, 
we can use the slide-back techniques of \cite{MR1666838} to strengthen this 
estimate,  showing there exist $T\in \R$, $a\in \R^n$, and $\beta > 1$
such that 
\begin{equation} \label{refined_asymp} 
\left | v(t,\theta) - v_\varepsilon(t+T) - e^{-t} \langle \theta, a \rangle 
\left ( -\dot v_\varepsilon (t+T) + \frac{n-2s}{2} v_\varepsilon(t+T) 
\right ) \right | = \mathcal{O} (e^{-\beta t}) \quad {\rm as} \quad t\to \infty. 
\end{equation}
Here, once again, the dot denotes a derivative with respect to $t$. 
We expect that this proof should be very similar to the one in \cite{MR1666838},
so we provide the strategy for the proof below and leave its details for subsequent work. 

We start with $v\in\mathcal{C}^\infty([0, \infty) \times \Ss^{n-1})$ a positive smooth solution to 
$$
P_s v + c_{n,s} v = c_{n,s} v^{\frac{n+2s}{n-2s}} \quad {\rm in} \quad [0, \infty) \times \Ss^{n-1}$$
and aim to show that there exists a Delaunay asymptote $v_\varepsilon$ and parameters
$T\in \R$, $a \in \R^n$, and $\beta > 1$ such that the estimate \eqref{refined_asymp} 
holds. 
Let $\{\tau_k\}_{k\in\mathbf{N}}\subset \mathbf{R}$ be any sequence of real numbers such that $\tau_k\to\infty$ as $k\to\infty$ and define the so-called slide-back sequence 
$\{v_k\}_{k\in\mathbf{N}}\subset \mathcal{C}^2([-\tau_k, \infty) \times \Ss^{n-1})$ as
\[v_k(t,\theta)
= v(t+\tau_k, \theta) \quad {\rm for} \quad k\in\mathbf{N}.
\]
By the {\it a priori} estimates proven in Section~\ref{sec:general_upper_bound}, 
we can use the Arzela-Ascoli theorem to extract a convergent subsequence $v_k 
\rightarrow v_*$, which, by our classification result, must have the form 
$$v_*(t,\theta) = v_\varepsilon (t+T).$$ 
To obtain simple asymptotics, we must show that both $\varepsilon,T>0$ do not depend 
on the choice of the sequence. We prove the independence of 
$\varepsilon$ using the radial Pohozaev invariant (Hamiltonian) in \eqref{def:DG-Hamiltonian} below; here 
it remains for us to prove that different Delaunay solutions in fact have different 
Pohozaev invariants. Next we prove the 
independence of $T$ follows from a delicate rescaling argument similar to the proof 
of \cite[Proposition 5]{MR1666838}. 

Once the simple asymptotics are established, the nondegeneracy sketched in the 
previous subsection together with a version of the linear decomposition lemma 
(see  \cite[Lemma 4.18]{MR1356375}) to 
get the refined asymptotics. 

\begin{remark}
The completeness condition implies the lower bound
$u(x)\gtrsim |x|^{(2s-n)/2}$, ruling out weakly singular behavior.
The asymptotic expansion~\eqref{refined_asymp} therefore provides a sharp description of all
isolated singularities for $s$ close to one.
\end{remark}

\subsubsection{Compactness of the moduli space of complete metrics}
As a final application, we describe a compactness result modeled on 
Pollack's compactness theorem \cite{MR1266101} in the scalar curvature 
setting. To state the result properly, we define both the marked 
moduli space and the unmarked moduli space of singular Yamabe metrics. 
If $\Lambda \subset \Ss^n$ is finite, we let 
\[
\mathcal{M}_{s,\Lambda} = \left \{ g=U^{\frac{4}{n-2s}}g_{\circ} \in [g_{\circ}]: 
\text{$U\in\mathcal{C}^\infty(\mathbf{S}^n\setminus \Lambda)$ is a positive singular solution to \eqref{const_q_sph_multiple}}\right \}/\sim ,
\]
Similarly, for $k \geqslant 3$, we define 
\[
\mathcal{M}_{s,k} = \left \{ g \in \mathcal{M}_{s,\Lambda} : 
\#\Lambda=k \right\}/\sim.
\]
Here, the quotient is taken over diffeomorphisms preserving the singular 
set $\Lambda\subset \mathbf{S}^n$ and the Gromov-Hausdorff topology is placed 
on both moduli spaces.

The compactness result we prove is the following: if $\Omega \subset \mathcal{M}_k$ 
is such that the radial Pohozaev invariants (Hamiltonian) of all Delaunay asymptotes are bounded away from 
zero and the distance between singular points remains bounded away from zero, 
then any sequence in $\Omega$ has a convergent subsequence. Again, we expect the 
proof to be similar in spirit to that of \cite{MR1266101}, \cite{MR4504918}, and 
\cite{cmpt_6th_order}, so we provide  only the central ideas of the proof. 

Since it plays such a central role, we first take some time to define the 
(radial) Pohozaev invariant as (see \cite[Theorem~1.2]{Azahara2018})\footnote{Please notice that there is a typo in the coefficients $e_1^*$ and $e^*$ as written in \cite{Azahara2018}, but the asymptotic behavior is correct.}
\begin{eqnarray} \label{def:DG-Hamiltonian} 
\mathcal{H}^*_s (v) = \frac{c_{n,s}}{\widetilde{d}_s} \left ( 
\frac{n-2s}{2n} v^{\frac{2n}{n-2s}} - \frac{1}{2} v^2 \right )  + \frac{1}{2} \int_0^{\rho_0^*} (\rho^*)^{1-2s} \left ( e_1^* (\rho) 
(\partial_t V)^2 - e^*(\rho) (\partial_{\rho^*} V)^2\right ) d\rho^*, 
\end{eqnarray} 
where $V\in \mathcal{C}^\infty(\mathbf{R}^2_+)$ is the Caffarelli-Silvestre extension of $v\in \mathcal{C}^\infty(\mathbf{R})$ and
 $e_1^*,e^* \in \mathcal{C}^\infty(0,\rho_0^*)$ are positive coefficients depending only on $\rho$ and the dimension such that $e_1^*(\rho), e_*(\rho)\to 1$ as $\rho\to 0$, where $\rho^*\in(0,\rho_0^*)$ with $\rho_0^*>0$ being the critical extension parameter given by \cite[Corollary~4.2]{Azahara2018}.
Here we recall that in Emden–Fowler coordinates, the bulk (Poincaré‐-Einstein) hyperbolic metric takes the form
\[
g_+
= d\rho^* + \bigl(1 + \tfrac{1-s}{2+2s}(\rho^*)^{n-1} \bigr) dt^2
+ \mathcal{O}(\rho^*)d t d\rho^*.
\]

We emphasize that the formula above holds only for the conformal factor expressed in log-cylindrical coordinates.
For the spherical and Euclidean gauge, we also set 
\[
\mathcal{H}^*_s (u):=(\mathcal{H}^*_s \circ \mathfrak{F}^{-1})(u) \quad {\rm and} \quad \mathcal{H}_s (U):=(\mathcal{H}^*_s \circ \mathfrak{F}^{-1})(u\circ \Pi),
\]
where $v\in \mathcal{C}^\infty(\mathbf{R})$, $u\in \mathcal{C}^\infty(\mathbf{R}^n\setminus \{0\})$, and $U\in \mathcal{C}^\infty(\mathbf{S}^{n}\setminus \{N,S\})$ are solutions to \eqref{eq_log-cyl}, \eqref{const_q_euc}, and \eqref{const_q_sph}, respectively.
Here we recall that
\[
\Pi:\mathbf{S}^n\setminus\{N,S\}\to \mathbf{R}^n\setminus\{0\} \quad {\rm and} \quad \mathfrak{F}:\mathcal{C}^\infty(\mathbf{R}^n\setminus\{0\})\to \mathcal{C}^\infty(\mathbf{R}\times\mathbf{S}^{n-1})
\]
are the stereographic projection and the log-cylindrical transform, respectively.

Let $\{ g_k= U_k^{\frac{4}{n-2s}}g_\circ\}$ be a sequence of singular Yamabe metrics 
such that all the asymptotic radial Pohozaev invariants are bounded away from zero, and 
the distances between singular points are also bounded away from zero. Using (for 
instance) \cite[Lemma 11]{MR4504918}, it suffices to assume the singular set is 
a fixed finite set $\Lambda$ for the entire sequence. 
Now, let $\{ D_\ell\}_{\ell\in\mathbf{N}}$ be a 
compact exhaustion of $\Ss^n \backslash \Lambda$ and for any large $\ell\gg1$, we extract a convergent subsequence $U_k \rightarrow U_*(\ell)$. Letting 
$\ell \rightarrow \infty$, we obtain a limit $U_*\in\mathcal{C}^{2,\alpha}(\Ss^n \backslash 
\Lambda)$ for some $\alpha\in(0,1)$. Finally,  
we use the fact that the Pohozaev invariants are bounded away from zero to 
show that $U_*>0$  is strictly positive and defines a complete metric on 
$\Ss^n \backslash \Lambda$,  {\it i.e.} $\lim_{p\to \Lambda}U(p)=\infty$. If $U_*\equiv 0$, we could 
rescale the sequence and contradict the positive lower bound for all Pohozaev 
invariants. 

In the regime in which we've established our classification, namely $1-\delta 
< s < 1$, we expect the proof described above to carry over from the local setting
to the nonlocal setting without any complications. It is worth remarking that 
in the local setting, one can write the radial Pohozaev invariant 
very explicitly in terms of the necksize, while such a relation in the nonlocal 
setting is more elusive, mostly because the expression \eqref{def:DG-Hamiltonian} is much 
more complicated. Thus, it would be very interesting to relate the necksize more 
explicitly to the radial Pohozaev invariant in the nonlocal setting. 

\section*{Declarations}

\subsection*{Funding}
J.H.A. acknowledges financial support from S\~ao Paulo Research Foundation (FAPESP) \#2020/07566-3, \#2021/15139-0, and \#2023/15567-8 and National Council for Scientific and Technological Development (CNPq) \#409764/2023-0, \#443594/2023-6, \#441922/2023-6, and \#306014/2025-4.
A. DlT. acknowledges financial support from the Spanish Ministry of Science and Innovation (MICINN), through the IMAG-Maria de Maeztu Excellence Grant \#CEX2020-001105-M/AEI/ 10.13039/501100011033 and FEDER-MINECO Grants \#PID2021- 122122NB-I00, \#PID2020-113596GB-I00 , and \#PID2024-155314NB-I00; RED2022-134784-T, funded by MCIN/AEI/10.13039/501100011033 and by J. Andalucia (FQM-116), Fondi Ateneo – Sapienza Universit`a di Roma, PRIN (Prot. 20227HX33Z) and INdAM-GNAMPA Projects 2023, 2024 and 2025 – CUP \#E53C2200193000, CUP \#E53C23001670001 and CUP \#E532400195000. 
J. M. do \'O acknowledges financial support from National Council for Scientific and Technological Development (CNPq) \#312340/2021-4, \#409764/2023-0, \#443594/2023-6, Coordenação de Aperfeiçoamento de Pessoal de Nível Superior (CAPES) MATH AMSUD \#88887.878894/2023-00, and Para\'iba State Research Foundation (FAPESQ), \# 3034/2021.  
J.R. acknowledges financial support from the Deutsche Forschungsgemeinschaftgrant (DFG) \#561401741.
J.W. acknowledges financial support from Hong Kong General Research Fund \#14309824 entitled ``New frontiers in singular limits of nonlinear partial differential equations''.

\subsection*{Conflict of interest}
The authors have no relevant financial or non-financial interests to disclose.

\subsection*{Data availability}
Data sharing is not applicable as no datasets were generated or analyzed during the study.

\subsection*{Ethics approval}
Not applicable.

\subsection*{Tool and computational resource disclosure}
Earlier versions of this paper were entirely human-generated, but had a gap. Eventually 
we consulted Claude Fable 5.2 to close this gap. The Large Language Model (LLM) suggested 
that we reformulate the operator $P_s$, given in \eqref{cyl_GJMS1} as \eqref{cyl_GJMS_sym}, 
which is an equivalent formulation taking advantage of the symmetry of the kernel $K_s$. The 
LLM further suggested the statements of Propositions \ref{prop:uniform_bound} 
and \ref{prop:cauchy_uniq} and drafted a proof of Proposition \ref{prop:cauchy_uniq}; 
this proof mostly recycled the analysis we'd already done in the previous versions of the 
paper. On the other hand, 
the proof we give of Proposition \ref{prop:uniform_bound} is entirely ours. The remaining 
differences between the present version of the paper and previous versions are also 
entirely ours. We take full responsibility for the mathematical 
contents of the paper. 

% \bibliography{references}
% \bibliographystyle{abbrv}

\end{document}